\documentclass{amsart}
\usepackage[all]{xy}
\usepackage{color}
\usepackage{amsthm}
\usepackage[colorlinks=true]{hyperref}

\numberwithin{equation}{subsection}

\newtheorem{theorem}{Theorem}[section]
\newtheorem*{theorem*}{Theorem}
\newtheorem{lemma}[theorem]{Lemma}
\newtheorem{proposition}[theorem]{Proposition}
\newtheorem{corollary}[theorem]{Corollary}
\newtheorem*{corollary*}{Corollary}

\theoremstyle{remark}
\newtheorem{definition}[theorem]{Definition}
\theoremstyle{remark}
\newtheorem{example}[theorem]{Example}
\newtheorem*{example*}{Example}

\theoremstyle{remark}
\newtheorem{remark}[theorem]{Remark}
\theoremstyle{remark}
\newtheorem{notation}[theorem]{Notation}
\DeclareMathOperator{\Mot}{Mot}
\DeclareMathOperator{\hocolim}{hocolim}
\DeclareMathOperator{\colim}{colim}
\DeclareMathOperator{\incl}{incl}
\DeclareMathOperator{\id}{id}
\newcommand{\Ho}{\mathsf{Ho}}
\newcommand{\ko}{\: , \;}
\newcommand{\too}{\longrightarrow}
\newcommand{\dg}{\mathsf{dg}}
\newcommand{\dgHo}{\mathsf{H}^0}

\newcommand{\cA}{{\mathcal A}}
\newcommand{\cB}{{\mathcal B}}
\newcommand{\cC}{{\mathcal C}}
\newcommand{\cD}{{\mathcal D}}
\newcommand{\cE}{{\mathcal E}}
\newcommand{\cF}{{\mathcal F}}

\newcommand{\cM}{{\mathcal M}}

\newcommand{\cR}{{\mathcal R}}
\newcommand{\cS}{{\mathcal S}}
\newcommand{\cT}{{\mathcal T}}
\newcommand{\cU}{{\mathcal U}}

\newcommand{\cW}{{\mathcal W}}

\newcommand{\bbD}{\mathbb{D}}

\newcommand{\bbL}{\mathbb{L}}
\newcommand{\bbK}{I\mspace{-6.mu}K}
\newcommand{\bbR}{\mathbb{R}}
\newcommand{\bbT}{\mathbb{T}}
\newcommand{\bbN}{\mathbb{N}}
\newcommand{\bbZ}{\mathbb{Z}}

\newcommand{\op}{\mathsf{op}} 
\newcommand{\ie}{\textsl{i.e.}\ }

\newcommand{\loccit}{\textsl{loc.\,cit.}}

\newcommand{\Hqe}{\mathsf{Hqe}}
\newcommand{\Hec}{\mathsf{Hec}}
\newcommand{\Hmo}{\mathsf{Hmo}}

\newcommand{\sSet}{\mathsf{sSet}}
\newcommand{\Map}{\mathsf{Map}}
\newcommand{\Tri}{\mathsf{Tri}}
\newcommand{\tri}{\mathsf{tri}}
\newcommand{\perf}{\mathsf{perf}}
\newcommand{\pretr}{\mathsf{pre}\mbox{-}\mathsf{tr}}
\newcommand{\Cat}{\mathsf{Cat}} 
\newcommand{\CAT}{\mathsf{CAT}} 
\newcommand{\h}{\mathsf{H}}
\newcommand{\Kw}{\underline{K}}
\newcommand{\w}{\wedge}
\newcommand{\Hom}{\mathsf{Hom}} 
\newcommand{\rep}{\mathsf{rep}} 
\newcommand{\Loc}{\mathsf{L}}
\newcommand{\St}{\mathsf{St}}
\newcommand{\Spec}{\mathsf{Spec}}
\newcommand{\stab}{\mathsf{stab}}
\newcommand{\Fun}{\mathsf{Fun}} 
\newcommand{\dgcat}{\mathsf{dgcat}}
\newcommand{\Trdgcat}{\mathsf{Trdgcat}}
\newcommand{\HO}{\mathsf{HO}} 
\newcommand{\Madd}{\Mot^{\mathsf{add}}_{\dg}}
\newcommand{\Maddw}{\underline{\Mot}^{\mathsf{add}}_{\alpha}}
\newcommand{\Maddwzero}{\underline{\Mot}^{\mathsf{add}}_{\aleph_0}}

\newcommand{\Uadd}{\cU^{\mathsf{add}}_{\dg}}
\newcommand{\Uaddw}{\underline{\cU}^{\mathsf{add}}_{\alpha}}
\newcommand{\Uaddwzero}{\underline{\cU}^{\mathsf{add}}_{\aleph_0}}

\newcommand{\Mloc}{\Mot^{\mathsf{loc}}_{\dg}}
\newcommand{\Mlocvw}{\underline{\Mot}^{\mathsf{wloc}}_{\alpha}}
\newcommand{\Mlocw}{\underline{\Mot}^{\mathsf{loc}}_{\alpha}}
\newcommand{\Mlocwzero}{\underline{\Mot}^{\mathsf{loc}}_{\aleph_0}}

\newcommand{\Uloc}{\cU^{\mathsf{loc}}_{\dg}}
\newcommand{\Ulocvw}{\underline{\cU}^{\mathsf{wloc}}_{\alpha}}
\newcommand{\Ulocw}{\underline{\cU}^{\mathsf{loc}}_{\alpha}}

\newcommand{\ModZ}{\mathsf{Mod}\text{-}\bbZ}
\newcommand{\Spt}{\mathsf{Spt}}

\newcommand{\uHom}{\underline{\mathsf{Hom}}}
\newcommand{\HomC}{\uHom_{!}}
\newcommand{\HomAw}{\uHom_{\alpha}^{\mathsf{add}}}
\newcommand{\HomL}{\uHom^{\mathsf{loc}}_{\dg}}
\newcommand{\HomLvw}{\uHom_{\alpha}^{\mathsf{wloc}}}
\newcommand{\HomLw}{\uHom_{\alpha}^{\mathsf{loc}}}

\newcommand{\dgS}{\cS}
\newcommand{\dgD}{\cD}

\newcommand{\internalcomment}[1]{}

\title{Non-connective $K$-theory via universal invariants}
\author[D.-C. Cisinski and G. Tabuada]{Denis-Charles Cisinski and Gon{\c c}alo~Tabuada}

\address{LAGA\\
CNRS~(UMR 7539)\\
Universit\'e Paris~13\\
\hbox{Avenue~Jean-Baptiste~Cl\'ement}\\
93430 Villetaneuse\\France}
\email{cisinski@math.univ-paris13.fr}
\urladdr{http://www.math.univ-paris13.fr/~cisinski/}

\address{Departamento de Matem{\'a}tica e CMA, FCT-UNL\\
Quinta da Torre\\
2829-516 Caparica\\
Portugal}
\email{tabuada@fct.unl.pt}

\date{\today}

\begin{document}
\begin{abstract}
In this article, we further the study of higher $K$-theory of dg categories via universal
invariants, initiated in~\cite{Additive}.
Our main result is the  co-representability of non-connective $K$-theory
by the base ring in the ``universal localizing motivator''. As an application, we obtain for free higher Chern characters, resp. higher trace maps,
from non-connective $K$-theory to cyclic homology, resp. to topological Hochschild homology.
\end{abstract}

\maketitle

\tableofcontents

\section*{Introduction}
\subsection*{Dg categories}
A {\em differential graded (=dg) category}, over a commutative base ring $k$, is a category enriched over 
cochain complexes of $k$-modules (morphisms sets are such complexes)
in such a way that composition fulfills the Leibniz rule\,:
$d(f\circ g)=(df)\circ g+(-1)^{\textrm{deg}(f)}f\circ(dg)$.
Dg categories enhance and solve many of the technical problems inherent to triangulated categories; see Keller's ICM adress~\cite{ICM}. In {\em non-commutative algebraic geometry} in the sense of
Bondal, Drinfeld, Kapranov, Kontsevich, To{\"e}n, Van den Bergh, $\ldots$ \cite{BK}  \cite{Bvan} \cite{Drinfeld} \cite{Chitalk} \cite{Kontsevich} \cite{finMotiv} \cite{Toen}, they are considered as dg-enhancements of (bounded) derived categories of quasi-coherent sheaves on a hypothetic non-commutative space.
\subsection*{Additive/Localizing invariants}
All the classical (functorial) invariants, such as algebraic $K$-theory, Hochschild homology, cyclic homology,
and even topological Hochschild homology and topological cyclic homology (see \cite[\S10]{THH}),
extend naturally from $k$-algebras to dg categories. In a ``motivic spirit'', in order to study {\em all}
these classical invariants simultaneously, the notions of {\em additive} and {\em localizing}
invariant have been introduced in \cite{Additive}. These notions, that we now recall, make use of
the language of Grothendieck derivators~\cite{Grothendieck}, a formalism which allows us to state
and prove precise universal properties at the level of homotopy theories, and to dispense with many
of the technical problems one faces in using model categories; consult Appendix~\ref{appendix}.
For this introduction, it is sufficient to think of (triangulated) derivators as (triangulated)
categories which have all the good properties of the homotopy category of a (stable)
model category.

Let $E: \HO(\dgcat) \to \bbD$ be a morphism of derivators, from the pointed
derivator associated to the Morita model structure (see \S\ref{sub:Morita}), to
a strong triangulated derivator
(in practice, $\bbD$ will be the derivator associated to a stable model category $\cM$,
and $E$ will come from a functor $\dgcat\to\cM$ which sends derived Morita equivalences to weak equivalences in $\cM$).
We say that $E$ is an {\em additive invariant} if it commutes with filtered homotopy colimits,
preserves the terminal object and 
\begin{itemize}
\item[add)] sends {\em split exact sequences} (see~\ref{def:ses}) 
to direct sums. 
\end{itemize}
$$ 
\begin{array}{lcccr}
\xymatrix{
 \cA \ar[r]_{I} & \cB \ar@<-1ex>[l]_R
\ar[r]_P & \cC \ar@<-1ex>[l]_{S} 
} & &\mapsto & & [E(I)\,\, E(S)]: E(\cA)\oplus E(\cC) \stackrel{\sim}{\too} E(\cB)
\end{array}
$$
We say that $E$ is a {\em localizing invariant} if it satisfies the same conditions
as an additive invariant but with condition add) replaced by\,:
\begin{itemize}
\item[loc)] sends {\em exact sequences} (see~\ref{def:ses}) 
to distinguished triangles 
\end{itemize}
$$
\begin{array}{lcccr}
\cA \stackrel{I}{\too} \cB \stackrel{P}{\too} \cC &&\mapsto&&
E(\cA) \stackrel{E(I)}{\too} E(\cB) \stackrel{E(P)}{\too} E(\cC) \too E(\cA)[1]
\end{array}
$$
In \cite{Additive}, the second named author
have constructed the {\em universal additive invariant}
$$\Uadd:\HO(\dgcat) \too \Madd$$
and the {\em universal localizing invariant}
$$\Uloc: \HO(\dgcat) \too \Mloc\,.$$
Roughly, every additive (resp. localizing) invariant $E: \HO(\dgcat) \to \bbD$
factors uniquely through $\Uadd$ (resp. $\Uloc$); see Theorem~\ref{thm:loc}.

For instance, Waldhausen's $K$-theory of dg categories defines
a functor from $\dgcat$ to the model category of spectra $\Spt$
which sends derived Morita equivalences to stable homotopy
equivalences (see~\ref{not:Kneg}),
hence a morphism of derivator
$$K:\HO(\dgcat)\too \HO(\Spt)\, .$$
As $K$-theory preserves filtered (homotopy) colimits, it follows from
Waldhausen's additivity theorem that $K$ is an additive invariant.
Hence there is a unique homotopy colimit preserving morphism of triangulated derivators
$$K_{\mathsf{add}}:\Madd\too \HO(\Spt)$$
such that $K=K_{\mathsf{add}}\circ \Uadd$.

Notice that a localizing invariant is also an additive invariant, but the converse does not hold.
Because of this universality property, which is a reminiscence of motives,
$\Madd$ is called the {\em additive motivator} and $\Mloc$ the {\em localizing motivator}.
Recall that they are both triangulated derivators.

Before going further, note that any triangulated derivator is canonically
enriched over spectra; see \S\ref{sub:Stab}. The spectra of morphisms in a triangulated
derivator $\bbD$ will be denoted by $\bbR\Hom(-,-)$. The main result from \cite{Additive} is the following\,:
\begin{theorem*}{(\cite[Thm.\,15.10]{Additive})}\label{thm:co-repadd}
For every dg category $\cA$, we have a natural isomorphism in the stable homotopy category
of spectra
$$K(\cA)\simeq \bbR\Hom(\,\Uadd(\underline{k}),\, \Uadd(\cA)\,) \,,$$
where $\underline{k}$ is the dg category with a single object and $k$
as the dg algebra of endomorphisms, and $K(\cA)$ is Waldhausen's $K$-theory
spectrum of $\cA$.
\end{theorem*}
However, from the ``motivic'' point of view,
this co-representability result is not completely satisfactory, in the sense
that all the classical invariants (except Waldhausen's $K$-theory)
are localizing.
Therefore, the base category $\Mloc(e)$ of the localizing motivator is
morally what we would like to consider as the triangulated category of
{\em non-commutative motives}.

From this point of view,
the importance of the computation of the (spectra of) morphisms between two objects
in the localizing motivator is by now clear. Hence, a fundamental
problem of the theory of non-commutative motives is to compute the localizing
invariants (co-)represented by dg categories in $\Mloc$.
In particular, we would like a co-representability theorem analogous to
the preceding one, with $\Madd$ replaced by $\Mloc$ and $K(\cA)$
replaced by the non-connective $K$-theory spectrum $\bbK(\cA)$
(see Notation~\ref{not:Kneg}).
The main result of this paper is a proof of the co-representability
of non-connective $K$-theory by the base ring $k$\,:

\begin{theorem*}{(Thm.\,\ref{thm:main})}\label{thm:main-intro}
For every dg category $\cA$, we have a natural isomorphism in the stable homotopy category
of spectra
$$\bbK(\cA)\simeq\bbR\Hom(\,\Uloc(\underline{k}),\, \Uloc(\cA)\,) \,.$$
In particular, we obtain isomorphisms of abelian groups
$$\bbK_n(\cA)\simeq\Hom(\,\Uloc(\underline{k})[n],\, \Uloc(\cA)\,)\,,\,\,n \in \bbZ\, .$$
\end{theorem*}
The above theorem asserts that non-connective $K$-theory is ``non-commutative motivic cohomology''.

Non-connective $K$-theory goes back to the work of Bass~ \cite{Bass},
Karoubi~\cite{Karoubi}, Pedersen-Weibel~\cite{Pedersen,Weibel},
Thomason-Trobaugh~\cite{Thomason}, Schlichting~ \cite{Marco},
and has been the source of many
deep results ``amalgamated'' with somewhat ad-hoc constructions.
The above co-representability theorem radically changes this state of affairs by offering,
to the best of the authors knowledge, the first conceptual characterization of non-connective
$K$-theory in the setting of dg categories. It provides a completely new understanding of non-connective $K$-theory
as the universal construction, with values in a stable context, which preserves filtered
homotopy colimits and satisfies the localization property. More precisely, the co-representability theorem above should be understood
as the construction of the universal Chern character from non-connective $K$-theory\,:
it follows from this result and from the Yoneda lemma
that, for any localizing invariant $E:\HO(\dgcat)\to\HO(\Spt)$ with values in the
triangulated derivator of spectra, the data of natural maps $\bbK(\cA)\to E(\cA)$
are equivalent to the datum of a single class in the stable homotopy group $E_0(k)=\pi_0(E(\underline{k}))$;
see Theorem \ref{thm:abschern}. This is illustrated as follows.

\subsection*{Applications}
Let 
$$ \bbK_n(-), HC_j(-), THH_j(-): \Ho(\dgcat) \too \ModZ\,,\,\, n \in \bbZ,  j \geq 0$$
be respectively, the $n$-th algebraic $K$-theory group functor, the $j$-th cyclic homology group functor, and
the $j$-th topological Hochschild homology group functor; see Section~\ref{chap:Chern}.
The co-representability theorem above furnishes for free higher Chern characters and trace maps\,: 
\begin{theorem*}{(Thm.\,\ref{thm:Chern})}
We have the following canonical morphisms of abelian groups for every small dg category $\cA$\,:
\begin{itemize}
\item[(i)] Higher Chern characters
$$ ch_{n,r}: \bbK_n(\cA) \too HC_{n+2r}(\cA)\,,\,\,n \in \bbZ\,,\,\, r \geq 0\,,$$
such that $ch_{0,r}: \bbK_0(k) \too HC_{2r}(k)$ sends $1\in \bbK_0(k)$ to
a generator of the $k$-module of rank one $HC_{2r}(k)$.
\item[(ii)] When $k=\bbZ$, higher trace maps
$$ tr_{n}: \bbK_n(\cA) \too THH_{n}(\cA)\,, \, n \in \bbZ\, ,$$
such that $tr_{0}: \bbK_0(k) \too THH_{0}(k)$ sends $1\in\bbK_0(\bbZ)$
to $1\in THH_0(\bbZ)$, and
$$ tr_{n,r}: \bbK_n(\cA) \too THH_{n+2r-1}(\cA)\,, \, n \in \bbZ\,,\,\, r \geq 1\,,$$
such that $tr_{0,r}: \bbK_0(k) \too THH_{2r-1}(k)$ sends $1\in\bbK_0(\bbZ)$
to a generator in the cyclic group $THH_{2r-1}(\bbZ)\simeq\bbZ/ r\bbZ$.
\item[(iii)] When $k=\bbZ/p\bbZ$, with $p$ a prime number, higher trace maps
$$ tr_{n,r}: \bbK_n(\cA) \too THH_{n+2r}(\cA)\,,\, n,r \in \bbZ\,,$$
such that $tr_{0,r}: \bbK_0(k) \too THH_{2r}(k)$ sends $1\in\bbK_0(\bbZ)$
to a generator in the cyclic group $THH_0(\bbZ/p\bbZ)\simeq\bbZ/ p\bbZ$.
\end{itemize}
\end{theorem*}
Before explaining the main ideas behind the proof of the co-repre\-sen\-tability theorem, we would like to emphasize that, as condition loc)
is much more subtle than condition add), the tools and arguments used in the
proof of co-representability result in the additive motivator, {\em are not available}
when we work with the localizing motivator.
\subsection*{Strategy of the proof}
The proof is based on a rather ``orthogonal'' construction of the
localizing motivator $\Mloc$. A crucial ingredient is the observation
that we should not start with the notion of derived Morita equivalence, but
with the weaker notion of quasi-equiconic dg functor (these are the
dg functors which are homotopically fully-faithful and essentialy surjective
after pre-triangulated completion):
in other words, this means that we should work first
with pre-triangulated dg categories which are not necessarily idempotent complete;
see Proposition~\ref{prop:fibquasiec}. The reason for this is that
negative $K$-theory is, in particular, an obstruction for the (triangulated)
quotients to be Karoubian (see~\cite{Marco}), which is a difficult aspect
to see when we work with the notion of derived Morita equivalence.
The other difficulty consists in understanding
the localizing invariants associated to additive ones\,: Schlichting's construction
of the non-connective $K$-theory spectrum is general enough to be applied to
a wider class of functors than connective $K$-theory, and we would like to use it
to understand as explicitely as possible $\Mloc$ as a left Bousfield localization of
$\Madd$. One crucial step in Schlichting's construction is
the study of a flasque resolution (of dg categories), obtained
as a completion by countable sums. The problem with this completion functor
is that it doesn't preserves filtered (homotopy) colimits.
Our strategy consists in skirting this problem by considering
invariants of dg categories which only preserve $\alpha$-filtered homotopy
colimits, for a big enough cardinal $\alpha$. Through a careful analysis of Schlichting's construction,
we prove a co-representability result in this wider setting, and then use the nice
behaviour of non-connective $K$-theory to solve the problem with filtered homotopy colimits at the end.

In particular, we shall work with quite a few different kinds of K-theories.
Given a small dg category $\cA$, there are two ways to complete it. First, one can
consider its pre-triangulated envelope; see Proposition~\ref{prop:fibquasiec}. Waldhausen's $K$-theory of the pre-triangulated envelope of $\cA$ will be denoted by $\underline{K}(\cA)$; see Notation~\ref{not:Ktri}. One can also consider the Morita envelope of $\cA$, that roughly speaking is the pseudo-abelian completion of the pre-triangulated envelope of $\cA$; see Proposition~\ref{prop:fibMorita}. We will write $K(\cA)$ for the Waldhausen $K$-theory of the Morita envelope of $\cA$; see Notation~\ref{not:Kperf}. This corresponds to what we usualy call $K$-theory of $\cA$; see the example below. We thus have a natural comparison map
$$\underline{K}(\cA)\too K(\cA)$$
which induces isomorphisms $\underline{K}_i(\cA)\simeq K_i(\cA)$ for $i>0$,
and a monomorphism $\underline{K}_0(\cA)\subset K_0(\cA)$. And of course, we shall
deal also with non-connective $K$-theory $\bbK(\cA)$ of $\cA$; see Notation~\ref{not:Kneg}. By construction, we have a comparison map
$$K(\cA)\too\bbK(\cA)$$
which induces isomorphisms $K_i(\cA)\simeq \bbK_i(\cA)$ for $i\geq 0$. See~\S~\ref{sub:K-theory} for further details.

\begin{example*}
Any $k$-algebra $A$ can be seen as a dg category with one object whose
endomorphisms are given by $A$ (seen as a complex of $k$-modules concentrated in degree $0$).
Then $K(A)$ (resp. $\underline{K}(A)$) is (equivalent to) the spectrum (seen
as an infinite loop space) obtained as the loop space of
Quillen's  $Q$-construction applied to the exact category of projective right $A$-modules
of finite type (resp. to the exact category of free $A$-modules of finite type).
In particular $K(A)$ agrees with what we usualy call the $K$-theory of $A$, which is not
the case of $\underline{K}(A)$ in general.
Finally, the negative (stable) homotopy groups of $\bbK(A)$ are
(isomorphic to) the usual negative $K$-groups of $A$ (as considered by Bass, Karoubi, Weibel, etc.).
\end{example*}

Here is a more detailed account on the contents of the paper\,: in the first two
Sections we recall some basic results and constructions concerning dg categories.

Let $\Trdgcat$ be the derivator associated with the quasi-equiconic model
structure, $\alpha$ a regular cardinal and $\cA$ a (fixed) dg category.
In Section~\ref{chap:additive}, we construct the {\em universal $\alpha$-additive invariant}
$$
\begin{array}{lcccccccccr}
\Uaddw:\Trdgcat \too \Maddw & &&&&&&&& & (\mbox{see Theorem}~\ref{thm:Uadd})
\end{array}
$$
(\ie $\Uaddw$ sends quasi-equiconic dg functors to isomorphisms, preserves the terminal
object as well as $\alpha$-filtering homotopy colimits, satisfies condition
add), and is universal for these
properties). We then prove the identification of spectra
$$
\begin{array}{lcccccr}
 \bbR\Hom(\,\Uaddw(\underline{k}),\, \Uaddw(\cA)\,) \simeq \Kw(\cA) &&&&&&
 (\mbox{see Theorem}~\ref{thm:KTradd})\,.
\end{array}
$$
In Section~\ref{chap:localizing}, we introduce the class of {\it strict exact sequences}
(see~\ref{def:Wald})\,:
$$
\cA \too \cB \too \cB/\cA\,, 
$$
where $\cB$ is a pre-triangulated dg category, $\cA$ is
a \emph{thick} dg subcategory of $\cB$, and $\cB/\cA$ is the Drinfeld
dg quotient~\cite{Drinfeld} of $\cB$ by $\cA$. We localize
the derivator $\Maddw$ to force the strict exact sequences to become distinguished triangles,
and obtain a homotopy colimit preserving morphism of triangulated derivators
$$\gamma_!:\Maddw\too\Mlocvw\, .$$
Using Theorem~\ref{thm:KTradd} and Waldhausen's fibration Theorem we obtain\,:
$$
\begin{array}{lcccccr}
 \bbR\Hom(\,\Ulocvw(\underline{k}),\, \Ulocvw(\cA)\,) \simeq \Kw(\cA) &&&&&& (\mbox{see Theorem}~\ref{thm:Kwloc})\,.
\end{array}
$$
Then, we localize the derivator $\Mlocvw$ so that the derived Morita equivalences become invertible,
which leads to a new homotopy colimit preserving morphism of triangulated
derivators $l_!:\Mlocvw\to\Mlocw$. The derivator $\Mlocw$ is the analog
of $\Mloc$ (in the case $\alpha=\aleph_0$, we have by definition $\Mlocw=\Mloc$).

In Section~\ref{chap:set-up}, we adapt Schlichting's construction
of non-connective $K$-theory of Frobenius pairs to the setting of dg categories (see Propositions
\ref{prop:dg-setup} and \ref{prop:comparison}). This is used
in Section~\ref{chap:negK} to construct a morphism of derivators
$V_l(-): \Trdgcat \to \Mlocvw$ for which, if $\alpha$ is big enough,
we have\,:
$$
\begin{array}{lcccccr}
\bbR\Hom(\,\Ulocvw(\underline{k}),\, V_l(\cA)\,) \simeq \bbK(\cA) &&&&&& (\mbox{see Proposition}~\ref{prop:negK})\,.
\end{array}
$$
A key technical point is the fact that $V_l(-)$ is an $\alpha$-localizing invariant;
see Proposition~\ref{prop:V-loc}. This allows us to prove the identification (still for $\alpha$
big enough)\,:
$$
\begin{array}{lcccccr} 
\bbR\Hom(\,\Ulocw(\underline{k}),\, \Ulocw(\cA)\,) \simeq \bbK(\cA) &&&&&& (\mbox{see Theorem}~\ref{thm:negK})\,.
\end{array}
$$
In Proposition~\ref{prop:locfiltered}, we prove that the localizing motivator $\Mloc$ can be obtained from $\Mlocw$ by localizing it with respect to the the morphisms of shape
$$
\underset{j \in J}{\mbox{hocolim}}\, \Ulocw(D_j) \too \Ulocw(\cA) \,,
$$ where $\cA$ is a $\alpha$-small dg cell (see \S\,\ref{sub:dg-cells}), $J$ is a $\alpha$-small filtered category, and $D: J \to \dgcat$ is a functor such that $\underset{j \in J}{\mbox{hocolim}}\, D_j\simeq \cA$. Let us now sum up the
steps of the construction of the universal localizing motivator\,:
$$\xymatrix{
\Trdgcat  \ar[rrr] \ar[d]_{\Uaddw} &&& \HO(\dgcat) \ar[d]^{\Uloc} \\
\Maddw \ar[r]^{\gamma_!} & \Mlocvw \ar[r]^{l_!} & \Mlocw \ar[r]_{} \ar[r] & \Mloc
}$$
Finally, using Theorem~\ref{thm:negK} and the fact that negative $K$-theory preserves
filtered colimits, we prove the co-representability Theorem.

The last Section is devoted to the construction of Chern characters and trace maps.
We have also included an Appendix containing several results on Grothendieck derivators,
used throughout the article and which are also of independent interest.

\section{Preliminaries}
\subsection{Notations}\label{sec:notations}
Throughout the article we will work over a fixed commutative base ring $k$. 

Let $\cC(k)$ be the category of complexes of $k$-modules. We will use cohomological
notation, \ie the differential increases the degree. The tensor product of complexes will be denoted
by $\otimes$. The category $\cC(k)$ is a symmetric monoidal model category, where one uses the
projective model structure for which the fibrations are the degreewise
surjections, while the weak equivalences are the quasi-isomorphisms.

Let $\sSet$ be the category of simplicial sets and $\Spt$ the
category of spectra; see~\cite{Bos-Fri} for details. Given a Quillen model
category $\cM$~\cite{Quillen}, we will denote by $\Ho(\cM)$ its homotopy category.

We will use freely the language and basic results of the theory of Grothendieck derivators.
A short reminder of our favourite tools and notations can be found in the Appendix.
If $\cM$ is a model category, we will denote by $\HO(\cM)$ its
associated derivator; see \S\ref{sub:leftBousfield}. We will write $e$ for the terminal category, so that for
a derivator $\bbD$, $\bbD(e)$ will be its underlying category (for instance, if
$\bbD=\HO(\cM)$, $\bbD(e)$ is the usual homotopy category of the model category $\cM$).

Every triangulated derivator $\bbD$ is canonically enriched over
spectra; see~\S\ref{sub:Stab}. We will denote by $\bbR\Hom_{\bbD}(X,Y)$ the spectrum of
maps from $X$ to $Y$ in $\bbD$. If there is not any ambiguity in doing so,
we will simplify the notations by writing $\bbR\Hom(X,Y)=\bbR\Hom_{\bbD}(X,Y)$.

Throughout the article the adjunctions will be displayed vertically with the left (resp. right) adjoint on the left (resp. right) hand-side.

\subsection{Triangulated categories}\label{sec:triangcat}
We assume the reader is familiar with the basic notions concerning triangulated
categories. The unfamiliar reader is invited to consult Neeman's book~\cite{Neeman},
for instance.
\begin{notation}\label{not:Tri}
We will denote by $\Tri$ the category of triangulated categories: the objects are the triangulated categories, and the maps are the exact functors.
\end{notation}
\begin{definition}{(\cite[4.2.7]{Neeman})}\label{def:compact}
Let $\cT$ be a triangulated category admitting arbitrary coproducts. An object $X$ in $\cT$ is
called {\em compact} if for each family $\{Y_i\}_{i \in I}$ of objects in $\cT$, the canonical morphism
$$ \underset{i \in I}{\bigoplus} \,\mathsf{Hom}_{\cT}(X,Y_i) \too
\mathsf{Hom}_{\cT}(X, \underset{i \in I}{\bigoplus} \,Y_i)$$
is invertible. Given a triangulated category $\cT$ admitting arbitrary
coproducts, we will denote by $\cT_c$ its full subcategory of compact objects
(note that $\cT_c$ is a thick subcategory of $\cT$; in particular, it has a natural
structure of triangulated category).
\end{definition}
\begin{definition}{(\cite[\S\,2]{Neeman})}\label{def:exact}
A sequence of triangulated categories
$$ \cR \stackrel{I}{\too} \cS \stackrel{P}{\too} \cT$$
is called {\em exact} if the composition $P\circ I$ is zero, while the functor
$\cR \stackrel{I}{\to} \cS$ is fully-faithful and the induced functor from
the Verdier quotient $\cS/\cR$ to $\cT$ is {\em cofinal},
$\ie$ it is fully-faithful and every object in $\cT$ is a direct summand of an object of $\cS/\cR$. 

An exact sequence of triangulated categories as above is
called {\em split exact}, if there exist triangulated functors
$$
\begin{array}{lcr}
\cS \stackrel{R}{\too} \cR & \mbox{and} & \cT \stackrel{S}{\too} \cS\,,
\end{array}
$$
with $R$ right adjoint to $I$, $S$ right adjoint to $P$ and $P\circ S=\mathrm{Id}_{\cT}$
and $R \circ I = \mathrm{Id}_{\cR}$ via the adjunction morphisms.
\end{definition}
\begin{notation}\label{not:idemp}
Given a triangulated category $\cT$, we will denote by $\widetilde{\cT}$ its
{\em idempotent completion}; see \cite{BM} for details. This construction
is functorial. Notice also that the idempotent completion of a split exact sequence is split exact.
\end{notation}
\section{Homotopy theories of dg categories}
In this Section, we will review the three homotopy theories of dg categories
developped in \cite[\S1-2]{Thesis}. This will give us the opportunity to fix some notation,
which will be used throughout the article. For a survey article on dg categories, the reader is advised
to consult Keller's ICM adress~\cite{ICM}. 
\begin{definition}\label{def:dgcategory}
A {\em small dg category $\cA$} is a $\cC(k)$-enriched category~\cite[Def.\,6.2.1]{Borceaux}.
Recall that this consists of the following data\,: 
\begin{itemize}
\item[-] a small set of objects $\mathrm{obj}(\cA)$ (usually denoted by $\cA$ itself); 
\item[-] for each pair of objects $(x,y)$ in $\cA$, a complex of $k$-modules $\cA(x,y)$; 
\item[-] for each triple of objects $(x,y,z)$ in $\cA$, a composition morphism in $\cC(k)$
$$\cA(y,z)\otimes \cA(x,y) \too \cA(x,z)\,,$$
satisfying the usual associativity condition; 
\item[-] for each object $x$ in $\cA$, a morphism $k \to \cA(x,x)$ in $\cC(k)$, satisfying the usual unit condition with respect to the above composition.
\end{itemize}
\end{definition}
\begin{definition}\label{def:dgfunctor}
A {\em dg functor} $F: \cA \to \cB$ is a $\cC(k)$-functor~\cite[Def.\,6.2.3]{Borceaux}. Recall that this consists of the following data\,:
\begin{itemize}
\item[-] a map of sets $F: \mathrm{obj}(\cA) \too \mathrm{obj}(\cB)$;
\item[-] for each pair of objects $(x,y)$ in $\cA$, a morphism in $\cC(k)$
$$ F(x,y): \cA(x,y) \too \cB(Fx, Fy)\,,$$
satisfying the usual unit and associativity conditions.
\end{itemize}
\end{definition}
\begin{notation}\label{not:dgcat}
We denote by $\dgcat$ the category of small dg categories.
\end{notation}
\subsection{Dg cells}\label{sub:dg-cells}
\begin{definition}\label{def:Icells}
\begin{itemize}
\item[-] Let $\underline{k}$ be the dg category with one
object $\ast$ and such that $\underline{k}(\ast,\ast):=k$ (in degree
zero). 
\item[-] For $n \in \mathbb{Z}$, let $S^{n}$ be the complex $k[n]$
(with $k$ concentrated in degree $n$) and let $D^n$ be the mapping
cone on the identity of $S^{n-1}$. We denote by $\dgS(n)$ the dg
category with two objects $1$ et $2$ such that $ \dgS(n)(1,1)=k \ko
\dgS(n)(2,2)=k \ko \dgS(n)(2,1)=0  \ko \dgS(n)(1,2)=S^{n} $ and
composition given by multiplication. 
\item[-] We denote by $\dgD(n)$ the dg
category with two objects $3$ and $4$ such that $ \dgD(n)(3,3)=k \ko
\dgD(n)(4,4)=k \ko \dgD(n)(4,3)=0 \ko \dgD(n)(3,4)=D^n $ and with
composition given by multiplication.
\item[-] Let $\iota(n):\dgS(n-1)\to \dgD(n)$ be the dg functor that sends $1$ to
$3$, $2$ to $4$ and $S^{n-1}$ to $D^n$ by the identity on $k$ in
degree~$n-1$\,:
$$
\vcenter{
\xymatrix@C=2em@R=1em{
\dgS(n-1) \ar@{=}[d] \ar[rrr]^{\displaystyle \iota(n)}
&&& \dgD(n) \ar@{=}[d]
\\
&&&\\
\\
1 \ar@(ul,ur)[]^{k} \ar[dd]_-{S^{n-1}}
& \ar@{|->}[r] &
& 3\ar@(ul,ur)[]^{k}  \ar[dd]^-{D^n}
\\
& \ar[r]^-{\incl} &
\\
2 \ar@(dr,dl)[]^{k}
& \ar@{|->}[r] &
& 4\ar@(dr,dl)[]^{k}
}}
\qquad\text{where}\qquad
\vcenter{\xymatrix@R=1em@C=.8em{ S^{n-1} \ar[rr]^-{\incl} \ar@{=}[d]
&& D^n \ar@{=}[d]
\\
\ar@{.}[d]
&& \ar@{.}[d]
\\
0 \ar[rr] \ar[d]
&& 0 \ar[d]
\\
0 \ar[rr] \ar[d]
&& k \ar[d]^{\id}
\\
k \ar[rr]^{\id} \ar[d]
&& k \ar[d]
&{\scriptstyle(\textrm{degree }n-1)}
\\
0 \ar[rr] \ar@{.}[d]
&& 0 \ar@{.}[d]
\\
&&}}
$$
\end{itemize}
\end{definition}
\begin{notation}\label{not:I}
We denote by $I$ the set consisting of the dg functors
$\{\iota(n)\}_{n\in \mathbb{Z}}$ and the dg functor $\emptyset
\rightarrow \underline{k}$ (where the empty dg category $\emptyset$
is the initial one).
\end{notation}
\begin{definition}{(\cite[Def.\,10.4.1]{Hirschhorn})}\label{def:alphasmall}
Let $\alpha$ be an infinite regular cardinal; see~\cite[Def.\,10.1.10]{Hirschhorn}. A dg category $\cA$ is called {\em $\alpha$-small} if for every regular cardinal $\lambda \geq \alpha$ and every $\lambda$-sequence of dg categories $\{\cB_{\beta}\}_{\beta < \lambda}$, the induced map of sets
$$ \underset{\beta< \lambda}{\mathrm{colim}}\, \Hom_{\dgcat}(\cA, \cB_{\beta}) \too \Hom_{\dgcat}(\cA, \underset{\beta< \lambda}{\mathrm{colim}}\, \cB_{\beta})$$
is invertible.
\end{definition}
\begin{definition}\label{def:dgcell}
A dg category $\cA$ is a \emph{dg cell} if
the (unique) map $\emptyset\to\cA$ is a transfinite composition of pushouts of
elements of $I$. If $\cA$ is moreover $\alpha$-small, we say that $\cA$ is a \emph{$\alpha$-small dg cell}
\end{definition}
\subsection{Dg modules}
\begin{definition}\label{def:opposite}
Let $\cA$ be a small dg category. The {\em
opposite dg category} $\mathcal{A}^{\op}$ of $\cA$ has the same objects
as $\mathcal{A}$ and complexes of morphisms given by
$\mathcal{A}^{\op}(x,y) =\mathcal{A}(y,x)$.
\end{definition}
\begin{definition}{(\cite[Def.\,3.1]{ICM})}\label{def:module}
Let $\cA$ be a small dg category. A {\em right dg $\cA$-module} (or simply a $\cA$-module) is a dg functor $M:\cA^{\op} \rightarrow
\cC_{\dg}(k)$ with values in the dg category $\cC_{\dg}(k)$ of
complexes of $k$-modules. 
\end{definition}
\begin{notation}\label{not:modules}
We denote by $\cC(\cA)$ (resp.\ by
$\cC_{\dg}(\cA)$) the category (resp.\ dg category) of $\cA$-modules. Notice that $\cC(\cA)$ is endowed with the projective Quillen model structure (see~\cite[Thm.\,3.2]{ICM}), whose weak equivalences and fibrations are defined objectwise. We denote by $\cD(\cA)$ the {\em derived category of $\cA$},
\ie the homotopy category $\Ho(\cC(\cA))$ or equivalently the localization of $\cC(\cA)$ with respect to the class of
quasi-isomorphisms. Notice that $\cD(\cA)$ is a triangulated category. 
\end{notation}
\begin{definition}\label{def:Yoneda}
The {\em Yoneda dg functor}
\begin{eqnarray*}
\underline{h}: \cA  \too   \cC_{\dg}(\cA) && x  \mapsto  \cA(-,x)=:\underline{h}(x)
\end{eqnarray*}
sends an object $x$ to the $\cA$-module $\cA(-,x)$ {\em represented by $x$}. 
\end{definition}
\begin{notation}\label{not:tri&perf}
\begin{itemize}
\item[-] Let $\tri(\cA)$  be the smallest triangulated subcategory of $\cD(\cA)$ (see Notation~\ref{not:modules}) which contains the $\cA$-modules $\underline{h}(x)$, $x \in \cA$. 
\item[-] Let $\perf(\cA)$ be the smallest thick triangulated subcategory of $\cD(\cA)$ (\ie stable under direct factors) which contains the $\cA$-modules $\underline{h}(x)$, $x \in \cA$.
\end{itemize}
\end{notation}
A dg functor $F: \cA \to \cB$ gives rise to a {\em restriction/extension of scalars} Quillen adjunction\,:
$$
\xymatrix{
\cC(\cB) \ar@<1ex>[d]^{F^{\ast}}\\
\cC(\cA) \ar@<1ex>[u]^{F_!}  \,.\\
}
$$
The functor $F^{\ast}$ preserves quasi-isomorphisms and the functor $F_!$ can be naturally derived. We obtain then the adjunction\,:
$$
\xymatrix{
\cD(\cB) \ar@<1ex>[d]^{F^{\ast}}\\
\cD(\cA) \ar@<1ex>[u]^{\bbL F_!} \,.
}
$$
\begin{remark}\label{rk:extension} 
Notice that the derived extension of scalars functor $\bbL F_!$ restricts to two triangulated functors\,: 
\begin{eqnarray*} \bbL F_!:\tri(\cA) \too \tri(\cB)& \mathrm{and} &\bbL F_!: \perf(\cA) \too \perf(\cB)\,.
\end{eqnarray*}
\end{remark}
\subsection{Quasi-equivalences}\label{sub:quasieq}
\begin{definition}\label{def:homotcat}
Let $\cA$ be a dg category. The category $\dgHo(\cA)$ has the same objects as $\cA$ and
morphisms given by $\dgHo(\cA)(x,y)=\textrm{H}^0(\cA(x,y))$, where $\textrm{H}^0$ denotes the $0$-th cohomology group. This construction is functorial and so we obtain a well-defined functor
\begin{eqnarray*}
\dgHo(-): \dgcat \too \mathrm{Cat} && \cA \mapsto \dgHo(\cA)\,,
\end{eqnarray*}
with values in the category of small categories.
\end{definition}
\begin{definition}\label{def:quasieq}
A dg functor $F:\cA \to \cB$ is a {\em quasi-equivalence} if\,:
\begin{itemize}
\item[(QE1)] for all objects $x,y \in \cA$, the morphism in $\cC(k)$
 $$F(x,y): \cA(x,y) \too \cB(Fx,Fy)$$
  is a quasi-isomorphism and
\item[(QE2)] the induced functor 
$$\dgHo(F): \dgHo(\cA) \too \dgHo(\cB)$$
is an equivalence of categories.
\end{itemize}
\end{definition}
\begin{theorem}{(\cite[Thm.\,1.8]{Thesis})}\label{thm:quasieq}
The category $\dgcat$ carries a cofibrantly generated Quillen model structure whose weak equivalences are the quasi-equivalences. The set of generating cofibrations is the set $I$ of Notation~\ref{not:I}.
\end{theorem}
\begin{notation}
We denote by $\Hqe$ the homotopy category hence obtained.
\end{notation}
\begin{definition}\label{def:tensorprod}
Given small dg categories $\cA$ and $\cB$, its {\em tensor product} $\cA \otimes \cB$ is defined as follows: the set of objects of $\cA \otimes \cB$ is $\mbox{obj}(\cA)\times \mbox{obj}(\cB)$ and for two objects $(x,y)$ and $(x',y')$ in $\cA \otimes \cB$, we define
$$ (\cA \otimes \cB)((x,y),(x',y'))= \cA(x,x') \otimes \cB(y,y')\,.$$
\end{definition}
The tensor product of dg categories defines a symmetric monoidal structure on $\dgcat$ (with unit the dg category $\underline{k}$ of Definition~\ref{def:Icells}), which is easily seen to be closed. However, the model structure of Theorem~\ref{thm:quasieq} endowed with this symmetric monoidal structure is {\em not} a symmetric monoidal model category, as the tensor product of two cofibrant objects in $\dgcat$ is not cofibrant in general. Nevertheless, the bifunctor $-\otimes-$ can be derived into a bifunctor
\begin{eqnarray*}
-\otimes^{\bbL}-\,: \Hqe \times \Hqe \too \Hqe && (\cA, \cB) \mapsto Q(\cA)\otimes \cB=\cA \otimes^{\bbL}\cB\,,
\end{eqnarray*}
where $Q(A)$ is a cofibrant resolution of $\cA$.
\begin{theorem}{(\cite[Thm.\,6.1]{Toen})}\label{thm:Toen}
The derived tensor monoidal structure $-\otimes^{\bbL}-$ on $\Hqe$ admits an
internal Hom-functor (denoted by $\bbR\underline{\Hom}(-,-)$ in \loccit)
$$\rep(-,-): \Hqe^{\op} \times \Hqe \too \Hqe\,.$$
\end{theorem}

\begin{remark}
The category $\mathsf{H}^0(\rep(\cA,\cB))$ can be described rather easily: this is
the full subcategory of $\cD(\cA^\op\otimes^\bbL\cB)$ spanned by bimodules $X$
such that, for any object $a$ of $\cA$, $X(a)$ belongs to the essential
image of the Yoneda embedding $\mathsf{H}^0(\cB)\to\cD(\cB)$.
\end{remark}

\subsection{Quasi-equiconic dg functors}\label{sub:equiconic}
\begin{definition}\label{def:quasiec}
A dg functor $F:\cA \to \cB$ is {\em quasi-equiconic} if the induced triangulated functor (see Remark~\ref{rk:extension}) 
$$\bbL F_!:\tri(\cA) \too \tri(\cB)$$
is an equivalence of categories.
\end{definition}
\begin{theorem}{(\cite[Thm.\,2.2]{Thesis})}\label{thm:quasiec}
The category $\dgcat$ carries a cofibrantly generated Quillen model structure (called the {\em quasi-equiconic model structure}) whose weak equivalences are the quasi-equiconic dg functors. The cofibrations are the same as those of Theorem~\ref{thm:quasieq}. 
\end{theorem}
\begin{notation}\label{not:quasiec}
We denote by $\Hec$ the homotopy category hence obtained.
\end{notation}
\begin{proposition}{(\cite[Prop.\,2.14]{Thesis})}\label{prop:Bousfield1}
The Quillen model structure of Theorem~\ref{thm:quasiec} is a left Bousfield localization of the Quillen model structure of Theorem~\ref{thm:quasieq}.
\end{proposition}
Notice that we have a well-defined functor (see Notations~\ref{not:Tri} and \ref{not:tri&perf})\,:
\begin{eqnarray*}
\tri:\Hec  \too  \Tri && \cA  \mapsto  \tri(\cA)\,.
\end{eqnarray*}
\begin{proposition}{(\cite[Prop.\,2.10]{Thesis})}\label{prop:fibquasiec}
The fibrant dg categories, which respect to the quasi-equiconic model structure (called the {\em pre-triangulated dg categories}), are the dg categories $\cA$ for which the image of the Yoneda embedding $\dgHo(\cA) \hookrightarrow \cD(\cA)$ (\ref{def:Yoneda}) is stable under (co)suspensions and cones. Equivalently, these are the dg categories $\cA$ for which we have an equivalence $\dgHo(\cA) \stackrel{\sim}{\to} \tri(\cA)$ of triangulated categories. 
\end{proposition}
\begin{proposition}{(\cite[Remark\,2.8]{Thesis})}\label{prop:repec}
The monoidal structure $-\otimes^{\bbL}-$ on $\Hqe$ descends to $\Hec$ and the internal Hom-functor $\rep(-,-)$ (see Theorem~\ref{thm:Toen}) can be naturally derived\,: 
\begin{eqnarray*}\rep_{tr}(-,-): \Hec^{\op} \times \Hec \too \Hec && (\cA, \cB) \mapsto \rep(\cA, \cB_f)=\rep_{tr}(\cA,\cB)\,,
\end{eqnarray*}
where $\cB_f$ is a pre-triangulated resolution of $\cB$ as in Proposition~\ref{prop:fibquasiec}.
\end{proposition}
\begin{remark}
The category $\mathsf{H}^0(\rep_{tr}(\cA,\cB))$ is
the full subcategory of $\cD(\cA^\op\otimes^\bbL\cB)$ spanned by bimodules $X$
such that, for any object $a$ of $\cA$, $X(a)$ belongs to the essential
image of the full inclusion $\tri(\cB)\to\cD(\cB)$.
\end{remark}
\subsection{Derived Morita equivalences}\label{sub:Morita}
\begin{definition}\label{def:Morita}
A dg functor $F: \cA \to \cB$ is a {\em derived Morita equivalence} if the induced triangulated functor (see Remark~\ref{rk:extension}) 
$$\bbL F_!: \perf(\cA) \too \perf(\cB)$$ 
is an equivalence of categories. 
\end{definition}
\begin{remark}\label{rk:Morita}
Since the triangulated categories $\cD(\cA)$ and $\cD(\cB)$ (see Notation~\ref{not:modules})
are compactly generated (see~\cite[\S\,8.1]{Neeman}) and the triangulated functor
\begin{equation}\label{eq:derextension}
\bbL F_!: \cD(\cA) \too \cD(\cB)
\end{equation}
preserves arbitrary sums, we conclude that a dg functor $F: \cA \to \cB$ is a derived Morita equivalence if and only if
the triangulated functor (\ref{eq:derextension}) is an equivalence.
\end{remark}
\begin{theorem}{(\cite[Thm.\,2.27]{Thesis})}\label{thm:Morita}
The category $\dgcat$ carries a cofibrantly generated model structure (called the {\em Morita model structure}), whose weak equivalences are the derived Morita equivalences. The cofibrations are the same as those of Theorem~\ref{thm:quasiec}. 
\end{theorem}
\begin{notation}\label{not:Morita}
We denote by $\Hmo$ the homotopy category hence obtained.
\end{notation}
\begin{proposition}{(\cite[Prop.\,2.35]{Thesis})}\label{prop:Bousfield2}
The Quillen model structure of Theorem~\ref{thm:Morita} is a left Bousfield
localization of the Quillen model structure of Theorem~\ref{thm:quasiec}.
\end{proposition}
Notice that we have a well-defined functor (see Notations~\ref{not:Tri} and \ref{not:tri&perf})\,:
\begin{eqnarray*}
\perf:\Hmo  \too  \Tri  && \cA  \mapsto  \perf(\cA)\,.
\end{eqnarray*}
\begin{proposition}{(\cite[Prop.\,2.34]{Thesis})}\label{prop:fibMorita}
The fibrant dg categories, which respect to the Morita model structure (called the
{\em Morita fibrant dg categories}), are the dg categories $\cA$ for which
the image of the Yoneda embedding $\dgHo(\cA) \hookrightarrow \cD(\cA)$
(\ref{def:Yoneda}) is stable under (co)suspensions, cones, and direct factors.
Equivalently, these are the dg categories $\cA$ for which we have an equivalence
$\dgHo(\cA) \stackrel{\sim}{\to} \perf(\cA)$ of (idempotent complete) triangulated categories.
\end{proposition}  
\begin{proposition}{(\cite[Remark\,2.40]{Thesis})}
The monoidal structure $-\otimes^{\bbL}-$ on $\Hec$ descends to $\Hmo$ and the internal Hom-functor $\rep_{tr}(-,-)$ (see Proposition~\ref{prop:repec}) can be naturally derived 
\begin{eqnarray*}\rep_{mor}(-,-): \Hmo^{\op} \times \Hmo \too \Hmo && (\cA, \cB) \mapsto \rep_{tr}(\cA, \cB_f)=\rep_{mor}(\cA,\cB)\,,
\end{eqnarray*}
where $\cB_f$ is a Morita fibrant resolution of $\cB$ as in Proposition~\ref{prop:fibMorita}.
\end{proposition}
\begin{remark}
The category $\mathsf{H}^0(\rep_{mor}(\cA,\cB))$ is
the full subcategory of $\cD(\cA^\op\otimes^\bbL\cB)$ spanned by bimodules $X$
such that, for any object $a$ of $\cA$, $X(a)$ belongs to the essential
image of the full inclusion $\perf(\cB)\to\cD(\cB)$.
\end{remark}
\subsection{$K$-theories}\label{sub:K-theory}
\begin{notation}{($\Kw$-theory)}\label{not:Ktri}
Let $\cA$ be a small dg category. We denote by $\tri^{\cW}(\cA)$
the full subcategory of $\cC(\cA)$ (see~\ref{not:modules})
whose objects are the $\cA$-modules which become isomorphic in $\cD(\cA)$
to elements of $\tri(\cA)$ (see~\ref{not:tri&perf}), and which are moreover
cofibrant. The category $\tri^{\cW}(\cA)$, endowed with the weak equivalences and
cofibrations of the Quillen model structure on $\cC(\cA)$, is a Waldhausen category;
see \cite[\S3]{DS} or \cite{ciskth}. We denote by $\Kw(\cA)$ the Waldhausen $K$-theory
spectrum~\cite{Wald} of $\tri^{\cW}(\cA)$. Given a dg functor $F: \cA \to \cB$, the
extension of scalars left Quillen functor $F_!: \cC(\cA) \to \cC(\cB)$ preserves weak equivalences,
cofibrations, and pushouts. Therefore, it restricts to an exact functor
$F_!: \tri^{\cW}(\cA) \to \tri^{\cW}(\cB)$ between Waldhausen categories.
Moreover, if $F$ is a quasi-equiconic dg functor, the Waldhausen functor $F_!$ is an equivalence.
We obtain then a well-defined functor with values in the homotopy category of spectra:
\begin{eqnarray*}
\Hec \too \Ho(\Spt) && \cA \mapsto \Kw(\cA)\,.
\end{eqnarray*}
\end{notation}
\begin{notation}{($K$-theory)}\label{not:Kperf}
Let $\cA$ be a small dg category. We denote by $\perf^{\cW}(\cA)$ the full subcategory
of $\cC(\cA)$ (see~\ref{not:modules}) whose objects are the $\cA$-modules
which become isomorphic in $\cD(\cA)$ to elements of $\perf(\cA)$ (see~\ref{not:tri&perf}),
and which are moreover cofibrant. The category $\perf^{\cW}(\cA)$, endowed with the weak
equivalences and cofibrations of the Quillen model structure on $\cC(\cA)$, is a Waldhausen
category; see \cite[\S3]{DS} or \cite{ciskth}. We denote by $K(\cA)$ the Waldhausen $K$-theory
spectrum~\cite{Wald} of $\perf^{\cW}(\cA)$. Given a dg functor $F: \cA \to \cB$, the extension of
scalars left Quillen functor $F_!: \cC(\cA) \to \cC(\cB)$ preserves weak equivalences, cofibrations,
and pushouts. Therefore, it restricts to an exact functor $F_!: \perf^{\cW}(\cA) \to \perf^{\cW}(\cB)$
between Waldhausen categories. Moreover, if $F$ is a derived Morita equivalence, the Waldhausen functor
$F_!$ is an equivalence. We obtain then a well-defined functor
with values in the homotopy category of spectra:
\begin{eqnarray*}
\Hmo \too \Ho(\Spt) && \cA \mapsto K(\cA)\,.
\end{eqnarray*}
\end{notation}
\begin{notation}{($\bbK$-theory)}\label{not:Kneg}
Let $\cA$ be a small dg category. We denote by $\bbK(\cA)$
the non-connective $K$-theory spectrum of $\perf^{\cW}(\cA)$; see~\cite[\S 12.1]{Marco}
(one can also consider Proposition \ref{compKnegdgFrob} below as a definition).
This construction induces a functor with values in the homotopy category of spectra:
\begin{eqnarray*}
\Hmo \too \Ho(\Spt) && \cA \mapsto \bbK(\cA)\,.
\end{eqnarray*}
\end{notation}
\section{Additive invariants}\label{chap:additive}
In this Section, we will re-prove the main
Theorems of the article \cite{Additive}, using the quasi-equiconic
model structure (see Theorem~\ref{thm:quasiec}) instead of the Morita
model structure; see Theorem~\ref{thm:Morita}. Moreover, and in contrast
with \cite{Additive}, we will work with a fixed infinite regular cardinal
$\alpha$ (see \cite[Def.\,10.1.10]{Hirschhorn})\,: we want to study additive
invariants which preserve only $\alpha$-filtered homotopy colimits of dg categories.

\begin{definition}\label{def:ses}
A sequence in $\Hec$, resp. in $\Hmo$, (see Notations~\ref{not:quasiec} and \ref{not:Morita})
$$ \cA \stackrel{I}{\too} \cB \stackrel{P}{\too} \cC$$
is called {\em exact} if the induced sequence of triangulated categories on the left, resp. on the right, (see Notation~\ref{not:tri&perf})
$$
\begin{array}{lccr}
\tri(\cA) \too \tri(\cB) \too \tri(\cC) && \perf(\cA) \too \perf(\cB) \too \perf(\cC)
\end{array}
$$
is exact (\ref{def:exact}). A {\em split exact sequence} in $\Hec$, resp. in $\Hmo$, is an exact
sequence in $\Hec$, resp. in $\Hmo$, which is equivalent to one of the form
$$ 
\xymatrix{
 \cA \ar[r]_{I} & \cB \ar@<-1ex>[l]_R
\ar[r]_P & \cC \ar@<-1ex>[l]_{S} \,,
}
$$
where $I$, $P$, $R$ and $S$ are dg functors, with $R$ right adjoint to $I$, $S$ right adjoint to $P$ and $P\circ S=\mathrm{Id}_{\cC}$ and $R\circ I=\mathrm{Id}_{\cA}$ via the adjunction morphisms.
\end{definition}
Let us denote by $\Trdgcat$ the derivator associated with the quasi-equiconic model structure (see Theorem~\ref{thm:quasiec}) on $\dgcat$. In particular, we have $\Trdgcat(e)=\Hec$.
\begin{theorem}\label{thm:Uadd}
There exists a morphism of derivators
$$ \Uaddw: \Trdgcat \too \Maddw\,,$$
with values in a strong triangulated derivator (\ref{def:srp}), which\,:
\begin{itemize}
\item[$\alpha$-flt)] commutes with $\alpha$-filtered homotopy colimits;
\item[$\underline{\mbox{p}}$)] sends the terminal object in $\Trdgcat$ to the terminal object in $\Maddw$ and
\item[\underline{\mbox{add}})] sends the split exact sequences in $\Hec$ (\ref{def:ses}) to direct sums 
$$\Uaddw(\cA)\oplus \Uaddw(\cC) \stackrel{\sim}{\too} \Uaddw(\cB)\,.$$
\end{itemize}
Moreover, $\Uaddw$ is universal with respect to these properties, $\ie$ for every strong triangulated derivator $\bbD$, we have an equivalence of categories
$$ (\Uaddw)^{\ast}: \HomC(\Maddw, \bbD) \stackrel{\sim}{\too} \HomAw(\Trdgcat, \bbD)\,,$$
where the right-hand side consists of the full subcategory of $\uHom(\Trdgcat, \bbD)$ of morphisms of derivators which verify the above three conditions.
\end{theorem}
\begin{notation}\label{not:AInv}
The objects of the category $\HomAw(\Trdgcat, \bbD)$ will be called {\em $\alpha$-additive invariants} and $\Uaddw$ the {\em universal $\alpha$-additive invariant}. 
\end{notation}
Before proving Theorem~\ref{thm:Uadd}, let us state a natural variation on
a result due to To\"en and Vaqui\'e \cite[Prop.\,2.2]{TV}, which
is verified by all three model structures on $\dgcat$ (see \S\,\ref{sub:quasieq}-\ref{sub:Morita}).
\begin{notation}
Given a Quillen model category $\cM$, we denote by
$$\Map(-,-):\Ho(\cM)^{\op}\times\Ho(\cM)\to\Ho(\sSet)$$
its homotopy function complex~\cite[Def.\,17.4.1]{Hirschhorn}.
\end{notation}
\begin{definition}
Let $\alpha$ be a infinite regular cardinal; see \cite[Def.\,10.1.10]{Hirschhorn}. An object $X$ in $\cM$ is said to be {\em homotopically $\alpha$-small} if, for any $\alpha$-filtered direct system $\{Y_j\}_{j \in J}$ in $\cM$, the induced map
$$\underset{j \in J}{\hocolim}\, \Map(X,Y_j) \too \Map(X,\underset{j \in J}{\hocolim}\, Y_j)$$
is an isomorphism in $\Ho(\sSet)$.
\end{definition}
\begin{proposition}\label{prop:general}
Let $\cM$ be a cellular Quillen model category, with $I$ a set of generating cofibration. If the (co)domains of the elements of $I$ are cofibrant, $\alpha$-small and homotopically $\alpha$-small,
then\,:
\begin{itemize}
\item[(i)] A $\alpha$-filtered colimit of trivial fibration is a trivial fibration.
\item[(ii)] For any $\alpha$-filtered direct system $\{X_j\}_{j \in J}$ in $\cM$, the natural morphism
$$ \underset{j \in J}{\hocolim}\,X_j \too \underset{j \in J}{\colim}\,X_j$$
is an isomorphism in $\Ho(\cM)$.
\item[(iii)] Any object $X$ in $\cM$ is equivalent to a $\alpha$-filtered colimit of $\alpha$-small $I$-cell objects.
\item[(iv)] A object $X$ in $\cM$ is homotopically $\alpha$-small presented if and only if it is equivalent to a retraction in $\Ho(\cM)$ of a $\alpha$-small $I$-cell object.
\end{itemize}
\end{proposition}
The proof of this Proposition is completely similar to the proof of \cite[Prop.\,2.2]{TV},
and so we omit it.
\begin{proof}{(of Theorem~\ref{thm:Uadd})}
The construction on $\Maddw$ is analogous to the construction of the additive motivator $\Madd$; see~\cite[Def.\,15.1]{Additive}. The only difference is that we start with the quasi-equiconic model structure instead of the Morita model structure and we consider $\alpha$-filtered homotopy colimits instead of filtered homotopy colimits.

Let us now guide the reader throughout the several steps of its construction. The quasi-equiconic model structure satisfies all the hypothesis of Proposition~\ref{prop:general}. Therefore we can construct, as in \cite[\S 5]{Additive}, the universal morphism of derivators which preserves $\alpha$-filtered homotopy colimits
$$ \bbR \underline{h}: \Trdgcat \too \mathsf{L}_{\Sigma} \mathsf{Hot}_{\dgcat_{\alpha}}\,.$$
Although the derivator $\Trdgcat$ is not pointed (the canonical dg functor $\emptyset \to 0$ is not quasi-equiconic), we can still construct a localization morphism
$$ \Phi: \mathsf{L}_{\Sigma} \mathsf{Hot}_{\dgcat_{\alpha}} \too \mathsf{L}_{\Sigma, \underline{\mathsf{p}}} \mathsf{Hot}_{\dgcat_{\alpha}}\,,$$
by a procedure analogous to the one of \cite[\S 6]{Additive}. We obtain the following universal property\,: for every pointed derivator $\bbD$, the composition $\Phi \circ \bbR\underline{h}$ induces (as in \cite[Prop.\,6.1]{Additive}) an equivalence of categories
$$ (\Phi \circ \bbR \underline{h})^{\ast}: \HomC(\mathsf{L}_{\Sigma, \underline{\mathsf{p}}} \mathsf{Hot}_{\dgcat_{\alpha}}, \bbD) \stackrel{\sim}{\too} \uHom_{\alpha\mbox{-}\mathsf{flt}, \underline{\mathsf{p}}}(\Trdgcat, \bbD)\,.$$
Now, since Sections \cite[\S 12-13]{Additive} can be entirely re-written using the quasi-equiconic model structure instead of the Morita model structure, we can localize $\mathsf{L}_{\Sigma, \underline{\mathsf{p}}} \mathsf{Hot}_{\dgcat_{\alpha}}$ as in \cite[14.5]{Additive}. We obtain then a localization morphism
$$ \mathsf{L}_{\Sigma, \underline{\mathsf{p}}} \mathsf{Hot}_{\dgcat_{\alpha}} \too \underline{\Mot}_{\alpha}^{\mathsf{unst}}\,,$$
with values in a ``Unstable Motivator''. Finally, we stabilize $\underline{\Mot}_{\alpha}^{\mathsf{unst}}$, as in \cite[\S 8]{Additive}, and obtain the universal $\alpha$-additive motivator $\Maddw$. The morphism $\Uaddw$ is given by the composition
$$\Trdgcat \stackrel{\Phi\circ \bbR\underline{h}}{\too} \mathsf{L}_{\Sigma, \underline{\mathsf{p}}} \mathsf{Hot}_{\dgcat_{\alpha}} \too \underline{\Mot}_{\alpha}^{\mathsf{unst}} \too \Maddw\,.$$
The proof of the universality of $\Uaddw$, with respect to the properties of Theorem~\ref{thm:Uadd}, is now entirely analogous to the one of \cite[Thm.\,15.4]{Additive}.
\end{proof}
\begin{proposition}\label{prop:GenAdd}
The set of objects $\{\,\Uaddw(\cB)[n]\,|\, n \in \bbZ\,\}$, with $\cB$ a $\alpha$-small dg cell (\ref{def:dgcell}), forms a set of compact generators (see~\cite[\S\,8.1]{Neeman}) of the triangulated category $\Maddw(e)$.
\end{proposition}
\begin{proof}
Recall from the proof of Theorem~\ref{thm:Uadd}, the construction of $\Maddw$ and $\Uaddw$. Notice that by construction the objects
$(\Phi \circ \bbR \underline{h})(\cB)$, with $\cB$ a $\alpha$-small dg cell, form a set of homotopically finitely presented generators of $\mathsf{L}_{\Sigma, \underline{\mathsf{p}}} \mathsf{Hot}_{\dgcat_{\alpha}}$. Since $\Maddw$ is obtained from $\mathsf{L}_{\Sigma, \underline{\mathsf{p}}} \mathsf{Hot}_{\dgcat_{\alpha}}$ using the operations of stabilization and left Bousfield localization with respect to a set of morphisms whose (co)domains are homotopically finitely presented, \cite[Lemma\,8.2]{Additive} and \cite[Lemma\,7.1]{Additive} allow us to conclude the proof.
\end{proof}

\begin{remark}\label{rk:enrichcomp}
Notice that, for a triangulated derivator $\bbD$ (\ref{def:srp}), and for an object $X$ of $\bbD(e)$,
the spectrum of maps functor $\bbR\Hom(X,-)$ preserves small sums if and only if $X$ is compact in the triangulated category $\bbD(e)$ (which means that the functors $\Hom_{\bbD(e)}(X[n],-)$ preserve small sums for
any integer $n$). Moreover, as the spectrum of maps functor $\bbR\Hom(X,-)$ preserves finite homotopy (co-)limits
for any object $X$, the spectrum of maps functor $\bbR\Hom(X,-)$ preserves small sums if and only if it preserves
small homotopy colimits. Hence $X$ is compact in  $\bbD$
(in the sense that $\bbR\Hom(X,-)$ preserves filtered homotopy colimits)
if and only if it is compact in the triangulated category $\bbD(e)$,
which is also equivalent to the property that
$\bbR\Hom(X,-)$ preserves arbitrary small homotopy colimits. We will use freely this
last characterization in the sequel of this article.

In particular, Proposition \ref{prop:GenAdd} can be restated by saying that
$\bbR\Hom(\Uaddw(\cB),-)$ preserves homotopy colimits for any $\alpha$-small dg cell $\cB$.
\end{remark}

\begin{theorem}\label{thm:KTradd}
For any dg categories $\cA$ and $\cB$, with $\cB$ a $\alpha$-small dg cell (\ref{def:dgcell}), we have a natural isomorphism in the stable homotopy category of spectra (see Proposition~\ref{prop:repec} and Notation~\ref{not:Ktri})
$$ \bbR\Hom(\,\Uaddw(\cB),\, \Uaddw(\cA)\,) \simeq \Kw(\rep_{tr}(\cB,\cA))\,.$$
\end{theorem}
\begin{proof}
The proof of this co-representability Theorem is entirely analogous to the one of \cite[Thm.\,15.10]{Additive}. Simply replace $\rep_{mor}(-,-)$ by $\rep_{tr}(-,-)$ and $K(-)$ by $\Kw(-)$ in (the proofs of) \cite[Prop.\,14.11]{Additive}, \cite[Thm.\,15.9]{Additive} and \cite[Thm.\,15.10]{Additive}.
\end{proof}
\section{Idempotent completion}
\begin{notation}\label{not:hat}
We denote by
\begin{eqnarray*}
(-)^{\w}: \dgcat  \too  \dgcat & & \cA  \mapsto  \cA^{\w}
\end{eqnarray*}
the Morita fibrant resolution functor (see Proposition~\ref{prop:fibMorita}), obtained by the small object argument~\cite[\S 10.5.14]{Hirschhorn}, using the generating trivial cofibrations (see \cite[\S\,2.5]{Thesis}) of the Morita model structure. Thanks to~\cite[Prop.\,2.27]{Thesis}, the functor $(-)^{\w}$ preserves quasi-equiconic dg functors and so it gives rise to a morphism of derivators
$$(-)^{\w}: \Trdgcat \too \Trdgcat\,.$$
Notice that by construction, we have also a $2$-morphism $\mbox{Id} \Rightarrow (-)^{\w}$ of derivators.
\end{notation}
\begin{proposition}\label{prop:hat-properties}
The morphism $(-)^{\w}$ of derivators preserves\,:
\begin{itemize}
\item[(i)] $\alpha$-filtered homotopy colimits,
\item[(ii)] the terminal object and
\item[(iii)] split exact sequences (\ref{def:ses}).
\end{itemize}
\end{proposition}
\begin{proof}
Condition (i) follows from Proposition~\ref{prop:general} (ii) and the fact that the (co)domains of the generating trivial cofibrations of the Morita model structure are $\alpha$-small. Condition (ii) is clear. In what concerns condition (iii), notice that we have a (up to isomorphism) commutative diagram (see Notations~\ref{not:Tri} and \ref{not:quasiec})\,: 
$$
\xymatrix{
\Hec \ar[d]_{\tri} \ar[rr]^{(-)^{\w}} && \Hec \ar[d]^{\tri} \\
\Tri \ar[rr]_{\widetilde{(-)}} && \Tri\,.
}
$$
Since the idempotent completion functor $\widetilde{(-)}$ (see Notation~\ref{not:idemp}) preserves split exact sequences, the proof is finished.
\end{proof}
By Proposition~\ref{prop:hat-properties}, the composition
$$ \Trdgcat \stackrel{(-)^{\w}}{\too} \Trdgcat \stackrel{\Uaddw}{\too} \Maddw$$
is a $\alpha$-additive invariant (\ref{not:AInv}). Therefore, by Theorem~\ref{thm:Uadd} we obtain an induced morphism of derivators and a $2$-morphism
$$
\begin{array}{lcr}
 (-)^{\w}: \Maddw \too \Maddw & & \mbox{Id} \Rightarrow (-)^{\w}\,,
\end{array} 
$$
such that
$ \Uaddw(\cA)^{\w} \simeq \Uaddw(\cA^{\w})$ for every dg category $\cA$.
\section{Localizing invariants}\label{chap:localizing}
In this Section, we will work with a fixed infinite regular cardinal $\alpha$ \cite[Def.\,10.1.10]{Hirschhorn}.
\subsection{Waldhausen exact sequences}
\begin{definition}\label{def:Wald}
Let $\cB$ be a pre-triangulated dg category, see Proposition~\ref{prop:fibquasiec}. A \emph{thick} dg subcategory of $\cB$
is a full dg subcategory $\cA$ of $\cB$ such that the induced functor
$\dgHo(\cA)\to\dgHo(\cB)$ turns $\dgHo(\cA)$ into a thick subcategory of $\dgHo(\cB)$
(which means that $\cA$ is pre-triangulated and that any object of
$\dgHo(\cB)$ which is a direct factor of an object of $\dgHo(\cA)$
is an object of $\dgHo(\cA)$). A \emph{strict exact sequence} is
a diagram of shape
$$
\begin{array}{rcl}
 \cA \too \cB \too \cB/\cA\,,
 \end{array}
 $$
in which $\cB$ is a pre-triangulated dg category, $\cA$ is
a thick dg subcategory of $\cB$, and $\cB/\cA$ is the Drinfeld
dg quotient~\cite{Drinfeld} of $\cB$ by $\cA$.
\end{definition}
\begin{remark}
Any split short exact sequence (\ref{def:ses}) is a strict exact sequence.
\end{remark}
\begin{notation}\label{not:genE}
We fix once and for all a \emph{set} $\cE$ of representatives
of homotopy $\alpha$-small thick inclusions, by which we mean that
$\cE$ is a set of full inclusions of dg categories $\cA\to\cB$
with the following properties\,:
\begin{itemize}
\item[(i)] For any inclusion $\cA\to\cB$ in $\cE$, the dg category $\cB$
is pre-triangulated, and $\cA$ is thick in $\cB$.
\item[(ii)] For any inclusion $\cA\to\cB$ in $\cE$, there exists an
$\alpha$-small dg cell $\cB_0$ (\ref{sub:dg-cells})
and a quasi-equiconic dg functor $\cB_0\to\cB$ (\ref{def:quasiec}).
\item[(iii)] For any $\alpha$-small dg cell $\cB_0$, there exists
a quasi-equiconic dg functor $\cB_0\to\cB$, with $\cB$ cofibrant
and pre-triangulated, such that any inclusion
$\cA\to\cB$ with $\cA$ thick in $\cB$ is in $\cE$.
\end{itemize}
\end{notation}
\begin{proposition}\label{prop:approx}
Let $\cB$ be a pre-triangulated dg category, and $\cA$ a thick dg subcategory of $\cB$ (\ref{def:Wald}).
Then there exists an $\alpha$-filtered direct system
$\{ \cA_j \stackrel{\epsilon_j}{\to}\cB_j\}_{j \in J}$
of dg functors, such that for any index $j$, we have
a commutative diagram of dg categories of shape
$$\xymatrix{
\cA'_j\ar[r]\ar[d]&\cB'_j\ar[d]\\
\cA_j\ar[r]&\cB_j\,,
}$$
in which the vertical maps are quasi-equivalences (\ref{def:quasieq}), the map $\cA'_j\to\cB'_j$ belongs to $\cE$ (see Notation~\ref{not:genE}), and moreover there exists an isomorphism in the homotopy category of arrows between dg categories
(with respect to the quasi-equiconic model structure of Theorem~\ref{thm:quasiec}) of shape
$$ \underset{j \in J}{\hocolim}\,\{ \cA_j \stackrel{\epsilon_j}{\too}\cB_j\}
\stackrel{\sim}{\too} (\cA \too \cB)\,.$$
\end{proposition}
\begin{proof}
By Proposition~\ref{prop:general}, there exists an $\alpha$-filtered direct system
of $\alpha$-small dg cells $\{\cB''_j\}_{j \in J}$ in $\Hec$, such that
$$ \underset{j \in J}{\hocolim}\,\cB''_j \stackrel{\sim}{\too} \cB\,.$$
By taking a termwise fibrant replacement $\{\cB_j\}_{j \in J}$
of the diagram $\{\cB''_j\}_{j \in J}$,
with respect to the quasi-equiconic model category structure (see Proposition~\ref{prop:fibquasiec}), we obtain
a $\alpha$-filtered diagram of pre-triangulated dg categories, and so an
isomorphism in $\Hec$ of shape
$$ \underset{j \in J}{\hocolim}\,\cB_j \stackrel{\sim}{\too} \cB\,.$$
Using the following fiber products
$$
\xymatrix{
\cA_j \ar[d] \ar[r]^{\epsilon_j} \ar@{}[dr]|{\ulcorner} & \cB_j \ar[d] \\
\cA \ar[r] & \cB\,,
}
$$
we construct an $\alpha$-filtered direct system
$\{ \cA_j \stackrel{\epsilon_j}{\to} \cB_j\}_{j \in J}$ such that
the map
$$ \underset{j \in J}{\hocolim}\,\cA_j \stackrel{\sim}{\too} \cA$$
is an isomorphism in $\Hec$. To prove this last assertion, we consider the
commutative diagram of triangulated categories
$$\xymatrix{
{\colim}\,\dgHo(\cA_j)\ar@{=}[r]\ar[d]
&\dgHo({\hocolim}\,\cA_j)\ar[r]\ar[d]
&\dgHo(\cA)\ar[d]\\
{\colim}\,\dgHo(\cB_j)\ar@{=}[r]
&\dgHo({\hocolim}\,\cB_j)\ar[r]
&\dgHo(\cB)\, ,
}$$
and use the fact that thick inclusions are stable under
filtered colimits in the category $\Tri$ of triangulated categories.

Now, for each index $j$, there exists a quasi-equiconic dg functor
$\cB''_j\to\cB'_j$ with $\cB'_j$ cofibrant and pre-triangulated,
such that any thick inclusion into $\cB'_j$ is in $\cE$.
As $\cB'_j$ and $\cB_j$ are isomorphic in $\Hec$, and as $\cB'_j$
is cofibrant and $\cB_j$ is pre-triangulated, we get a quasi-equivalence
$\cB'_j\to\cB_j$. In conclusion, we obtain pullback squares
$$\xymatrix{
\cA'_j\ar[r]\ar[d]\ar@{}[dr]|{\ulcorner}&\cB'_j\ar[d]\\
\cA_j\ar[r]&\cB_j
}$$
in which the vertical maps are quasi-equivalences, and the map $\cA'_j\to\cB'_j$ belongs to $\cE$.
\end{proof}

The derivator $\Maddw$ (see Theorem~\ref{thm:Uadd}) admits a (left proper and cellular) Quillen model and so we can consider its left Bousfield localization with respect to the set of maps
\begin{equation}\label{eq:default}
 \Theta_{\epsilon}: \mathsf{cone} [\,\Uaddw(\epsilon):\Uaddw(\cA) \too \Uaddw(\cB)\,] \too \Uaddw (\cB/\cA)\,,
\end{equation}
where $\epsilon:\cA\to\cB$ belongs to $\cE$ (see~\ref{not:genE}). We obtain then
a new derivator $\Mlocvw$ (which admits also a left proper and cellular Quillen model) and an adjunction
$$ 
\xymatrix{
\Maddw \ar@<-1ex>[d]_{\gamma_!} \\
\Mlocvw \ar@<-1ex>[u]_{\gamma^{\ast}}\,.
}
$$
The morphism $\gamma_!$ preserves homotopy colimits and sends the maps $\Theta_\epsilon$
to isomorphisms, and it is universal with respect to these properties; see \ref{def:Cis}.

\begin{theorem}\label{thm:veryweakloc}
The composition morphism
$$ \Ulocvw: \Trdgcat \stackrel{\Uaddw}{\too} \Maddw \stackrel{\gamma_!}{\too} \Mlocvw$$
has the following properties\,:
\begin{itemize}
\item[$\alpha$-flt)] it commutes with $\alpha$-filtered homotopy colimits;
\item[$\underline{\mbox{p}}$)] it preserves the terminal object;
\item[$\underline{\mbox{wloc}}$)] it sends strict exact sequences (\ref{def:Wald})
to distinguished triangles in $\Mlocvw$.
\end{itemize}
Moreover $\Ulocvw$ is universal with respect to these properties, $\ie$ for every strong triangulated derivator $\bbD$, we have an equivalence of categories
$$ (\Ulocvw)^{\ast}: \HomC(\Mlocvw, \bbD) \stackrel{\sim}{\too} \HomLvw(\Trdgcat, \bbD)\,,$$
where the right-hand side consists of the full subcategory of $\uHom(\Trdgcat, \bbD)$ of morphisms of derivators which verify the above three conditions.
\end{theorem}
\begin{proof}
Since $\gamma_!$ preserves homotopy colimits, the morphism $\Ulocvw$ satisfies conditions
$\alpha$-flt), $\underline{\mbox{p}}$) and $\underline{\mbox{add}}$)
stated in Theorem \ref{thm:Uadd}.
Condition $\underline{\mbox{wloc}}$) follows from Proposition~\ref{prop:approx}
and the fact that $\Uaddw$ commutes with $\alpha$-filtered homotopy colimits..
Finally, the universality of $\Ulocvw$ is a consequence of Theorem~\ref{thm:Uadd} and Definition~\ref{def:Cis} (notice that, as any split exact sequence is a strict exact sequence,
under condition $\underline{\mbox{p}}$), condition $\underline{\mbox{add}}$) is implied by condition $\underline{\mbox{wloc}}$)).
\end{proof}
Let us now prove a general result, which will be used in the proofs of Theorems~\ref{thm:Kwloc} and \ref{thm:main}.
\begin{proposition}\label{prop:genarg}
Let $\bbD$ be a stable derivator (\ref{def:srp}), $S$ a set of morphisms in the base category $\bbD(e)$ and $X$ a compact object in $\bbD(e)$. Let us consider the left Bousfield localization of $\bbD$ with respect to $S$ (see \ref{def:Cis})\,:
$$ 
\xymatrix{
\bbD \ar@<-1ex>[d]_{\gamma_!} \\
\Loc_S\bbD \ar@<-1ex>[u]_{\gamma^{\ast}}\,.
}
$$
If the functor
$ \bbR\Hom_{\bbD}(X,-): \bbD(e) \to \Ho(\Spt)$
sends the elements of $S$ to isomorphisms, then $\gamma_!(X)$ is compact in $\Loc_S\bbD(e)$ and for every object $M \in \bbD(e)$, we have a natural isomorphism in the homotopy category of spectra\,: 
$$\bbR\Hom_{\bbD}(X,M)\simeq \bbR\Hom_{\Loc_S\bbD}(\gamma_!(X), \gamma_!(M))\,.$$
\end{proposition}
\begin{proof}
Let us start by showing the above natural isomorphism in the homotopy category of spectra. Thanks to \cite[Prop.\,4.1]{Additive}, $\Loc_S\bbD(e)$ is the localization of $\bbD(e)$ with respect to the smallest class $\mathsf{W}$ of maps in $\bbD(e)$, which contains the class $S$, has the two-out-of-three property, and which is closed under homotopy colimits. Since $X$ is compact and the functor $\bbR\Hom_{\bbD}(X,-)$ sends the elements of $S$ to isomorphisms, we conclude that it also sends all the elements of $\mathsf{W}$ to isomorphisms. Finally, since for every object $M \in \bbD(e)$, the co-unit adjunction morphism $M \to \gamma^{\ast}\gamma_!(M)$ belongs to $\mathsf{W}$, we obtain the following natural isomorphism in the homotopy category of spectra\,:
$$
 \bbR\Hom_{\Loc_S\bbD}(\gamma_!(X), \gamma_!(M)) \simeq  \bbR\Hom_{\bbD}(X, \gamma^{\ast}\gamma_!(M))
 \simeq   \bbR\Hom_{\bbD}(X, M) \,.$$
We now show that the object $\gamma_!(X)$ is compact (\ref{def:compact}) in the triangulated category $\Loc_S\bbD(e)$. Notice that for every object $N \in \Loc_S\bbD(e)$, the unit adjunction morphism $\gamma_! \gamma^{\ast}(N) \stackrel{\sim}{\to} N$ is an isomorphism. Let $(N_i)_{i \geq 0}$ be a family of objects in $\Loc_S\bbD(e)$. We have the following equivalences
$$
\begin{array}{rcl}
\Hom_{\Loc_S \bbD(e)}(\gamma_!(X), \underset{i}{\oplus} N_i)  & \simeq & \Hom_{\Loc_S\bbD(e)}(\gamma_!(X), \underset{i}{\oplus}\, \gamma_!\gamma^{\ast}(N_i)) \\
& \simeq & \Hom_{\Loc_S\bbD(e)}(\gamma_!(X),  \gamma_!(\underset{i}{\oplus}\,\gamma^{\ast}(N_i)))\\
& \simeq & \Hom_{\bbD(e)}(X, \underset{i}{\oplus}\,\gamma^{\ast}(N_i))\\
& \simeq & \underset{i}{\oplus}\, \Hom_{\bbD(e)}(X, \gamma^{\ast}(N_i))\\
& \simeq & \underset{i}{\oplus}\, \Hom_{\Loc_S\bbD(e)}(\gamma_!(X), \gamma_!\gamma^{\ast}(N_i))\\
& \simeq & \underset{i}{\oplus}\, \Hom_{\Loc_S\bbD(e)}(\gamma_!(X), N_i)\\
\end{array}
$$
and so the proof is finished.
\end{proof}

\begin{theorem}[Waldhausen Localization Theorem]\label{thm:Kwloc}
The object $\Ulocvw(\underline{k})$ is compact (\ref{def:compact}) in the triangulated category $\Mlocvw(e)$ and for every dg category $\cA$, we have a natural isomorphism in the stable homotopy category of spectra (see Notation~\ref{not:Ktri})
$$ \bbR\Hom(\,\Ulocvw(\underline{k}),\, \Ulocvw(\cA)\,) \simeq \Kw(\cA)\,.$$
\end{theorem}
\begin{proof}
The proof consists on verifying the conditions of Proposition~\ref{prop:genarg}, with $\bbD=\Maddw$, $S=\{ \Theta_{\epsilon}\,|\, \epsilon \in \cE\}$ (see~\ref{eq:default}), $X= \Uaddw(\underline{k})$ and $M=\Uaddw(\cA)$. Since by Proposition~\ref{prop:GenAdd}, $\Uaddw(\underline{k})$ is compact in $\Maddw(e)$, it is enough to show that the functor
$$ \bbR\Hom(\Uaddw(\underline{k}), -): \Maddw(e) \too \Ho(\Spt)$$
sends the elements of $S$ to isomorphisms.  For this, consider the following diagram
$$
\xymatrix{
\Uaddw(\cA) \ar@{=}[d] \ar[rr]^{\Uaddw(\epsilon)} &&  \Uaddw(\cB) \ar@{=}[d] \ar[rr] && \mathsf{cone}(\Uaddw(\epsilon)) \ar[d]^{\Theta_{\epsilon}} \\
\Uaddw(\cA) \ar[rr]_{\Uaddw(\epsilon)} && \Uaddw(\cB) \ar[rr] && \Uaddw(\cB/\cA)\,.
}
$$
By virtue of Theorem~\ref{thm:KTradd} and Proposition~\ref{prop:GenAdd}, the functor $\bbR\Hom(\Uaddw(\underline{k}),-)$, applied to the above diagram,
gives the following one
$$
\xymatrix{
\Kw(\cA) \ar@{=}[d] \ar[rr]^{\Kw(\epsilon)} && \Kw(\cB) \ar@{=}[d]  \ar[rr] && \mathsf{cone}(\Kw(\epsilon)) \ar[d] \\
\Kw(\cA) \ar[rr]_{\Kw(\epsilon)} && \Kw(\cB) \ar[rr] && \Kw(\cB/\cA)\,,
}
$$
where the upper line is a homotopy cofiber sequence of spectra. Now, consider the following sequence of Waldhausen categories (see Notation~\ref{not:Ktri})
$$ \tri^{\cW}(\cA) \too \tri^{\cW}(\cB) \too \tri^{\cW}(\cB/\cA)\,.$$
Let us denote by $v\tri^{\cW}(\cB)$ the Waldhausen category with the same cofibrations as $\tri^{\cW}(\cB)$, but whose weak equivalences are the maps with $\mathsf{cone}$ in $\tri^{\cW}(\cA)$. Since $\tri(\cA)$ is thick in $\tri(\cB)$, we conclude by Waldhausen's fibration Theorem~\cite[Thm.\, 1.6.4]{Wald} (applied to the inclusion $\tri^{\cW}(\cB) \subset v\tri^{\cW}(\cB)$) that the lower line of the above diagram is a homotopy fiber sequence of spectra. Since the (homotopy) category of spectra is stable, we conclude that the right vertical map is a weak equivalence of spectra. Finally, since by Theorem~\ref{thm:KTradd}, we have a natural isomorphism in the stable homotopy category of spectra
$$ \bbR\Hom(\,\Uaddw(\underline{k}),\, \Uaddw(\cA)\,) \simeq \Kw(\cA)\,,$$
the proof is finished.
\end{proof}
\subsection{Morita invariance}
\begin{notation}
Let $\cS$ be the set of morphisms in $\Hec$ of the form $ \cB \to \cB^{\w}$ (see Notation~\ref{not:hat}), with $\cB$ a $\alpha$-small dg cell (\ref{def:dgcell}).
\end{notation}
\begin{proposition}\label{prop:aprox}
For any dg category $\cA$, there exists an $\alpha$-filtered direct system $\{ \cB_j \to \cB_j^{\w} \}_{j \in J}$ of elements of $\cS$, such that 
$$\underset{j \in J}{\hocolim}\,\{ \cB_j \too \cB_j^{\w} \} \stackrel{\sim}{\too} (\cA \too \cA^{\w})\,.$$
\end{proposition}
\begin{proof}
Thanks to~\ref{prop:general}\,(iii), there exits an $\alpha$-filtered direct system $\{\cB_j \}_{j \in J}$ of $\alpha$-small dg cells such that
$$ \underset{j \in J}{\hocolim}\, \cB_j \stackrel{\sim}{\too} \cA \,.$$
Since by Proposition~\ref{prop:hat-properties} (i), the functor $(-)^{\w}$ preserves $\alpha$-filtered homotopy colimits, the proof is finished.
\end{proof}
Since the derivator $\Mlocvw$ (see Theorem~\ref{thm:veryweakloc}) admits a (left proper and cellular) Quillen model, we can localize it with respect to the image of the set $\cS$ under $\Ulocvw$. We obtain then a new derivator $\Mlocw$ and an adjunction
$$
\xymatrix{
\Mlocvw \ar@<-1ex>[d]_{l_!} \\
\Mlocw \ar@<-1ex>[u]_{l^{\ast}}\,.
}
$$
\begin{proposition}\label{prop:MorInv}
The composition
$$\Ulocw: \Trdgcat \stackrel{\Ulocvw}{\too} \Mlocvw \stackrel{l_!}{\too} \Mlocw$$
sends the derived Morita equivalences (\ref{def:Morita}) to isomorphisms.
\end{proposition}
\begin{proof}
Observe first that $\Ulocw$ sends the maps of shape $\cA\to\cA^{\w}$
to isomorphisms. This follows from Proposition~\ref{prop:aprox} and from the fact that $\Ulocvw$ commutes with $\alpha$-filtered homotopy colimits. Thanks to Propositions~\ref{prop:Bousfield1} and \ref{prop:Bousfield2}, a dg functor $F:\cA \to \cB$ is a derived Morita equivalence if and only if $F^{\w}: \cA^{\w} \to \cB^{\w}$ is a quasi-equivalence (\ref{def:quasieq}). The proof is then finished.
\end{proof}
Thanks to Proposition~\ref{prop:MorInv}, the morphism $\Ulocw$ descends to the derivator $\HO(\dgcat)$ associated with the Morita model structure on $\dgcat$; see Theorem~\ref{thm:Morita}.
\begin{theorem}\label{thm:weakloc}
The morphism
$$ \Ulocw: \HO(\dgcat) \too \Mlocw$$
\begin{itemize}
\item[$\alpha$-flt)] commutes with $\alpha$-filtered homotopy colimits;
\item[$\underline{\mbox{p}}$)] preserves the terminal object and
\item[loc)] sends the exact sequences in $\Hmo$ (\ref{def:ses}) to distinguished triangles in $\Mlocw$.
\end{itemize}
Moreover $\Ulocw$ is universal with respect to these properties, $\ie$ for every strong triangulated derivator $\bbD$, we have an equivalence of categories
$$ (\Ulocw)^{\ast}: \HomC(\Mlocw, \bbD) \stackrel{\sim}{\too} \HomLw(\HO(\dgcat), \bbD)\,,$$
where the right-hand side consists of the full subcategory of $\uHom(\HO(\dgcat), \bbD)$ of morphisms of derivators which verify the above three conditions.
\end{theorem}
\begin{notation}\label{not:alphaLoc}
The objects of the category $\HomLw(\HO(\dgcat), \bbD)$ will be called {\em $\alpha$-localizing invariants} and $\Ulocw$ the {\em universal $\alpha$-localizing invariant}.
\end{notation}
\begin{proof}
Since $l_!$ preserves homotopy colimits, the morphism $\Ulocw$ satisfies conditions $\alpha$-flt) and $\underline{\mbox{p}}$). In what concerns condition loc), we can suppose by \cite[Thm.\,4.11]{ICM} that we have a exact sequence in $\Hmo$ of the form
$$ \cA \stackrel{I}{\too} \cB \stackrel{P}{\too} \cB/\cA \,.$$
Consider the following diagram
$$
\xymatrix{
\cA \ar[d] \ar[r]^I & \cB \ar[d] \ar[r] & \cB/\cA \ar[d] \\
\cA^{\w} \ar[r]_{I^{\w}} & \cB^{\w} \ar[r] &\cB^{\w}/\cA^{\w}\,,
}
$$
where the right vertical map is the induced one. Since the first two vertical maps are derived Morita equivalences, the right vertical one is also a derived Morita equivalence. Furthermore, since the lower line is a
strict exact sequence (\ref{def:Wald}), Proposition~\ref{prop:MorInv} and Theorem~\ref{thm:veryweakloc} imply that $\Ulocw$ satisfies condition loc). The universality of $\Ulocw$ is now a clear consequence of Theorem~\ref{thm:veryweakloc} and Proposition~\ref{prop:MorInv}.
\end{proof}
\section{Schlichting's set-up}\label{chap:set-up}
In this Section we will built a set-up (see~\cite[\S 2.2]{Marco}) in $\Hec$, that will be used in Section~\ref{chap:negK} to construct the non-connective $K$-theory spectrum. In Proposition~\ref{prop:comparison}, we will relate it with Schlichting's set-up on Frobenius pairs.
\subsection{Set-up}
We start by constructing a infinite sums completion functor. Recall from \cite[\S 6]{Sums}, the construction of the following Quillen adjunctions (where we have taken $\alpha=\aleph_1$)
$$
\xymatrix{
\dgcat_{ex,\aleph_1} \ar@<1ex>[d]^{U_1} \\
T_{\aleph_1}\mbox{-}\mathsf{alg} \ar@<1ex>[d]^U \ar@<1ex>[u]^{F_1} \\
\dgcat \ar@<1ex>[u]^F\,,
}
$$ 
where $\dgcat$ is endowed with the model structure of Theorem~\ref{thm:quasieq}. Consider the composed functor
\begin{equation}\label{eq:composition}
\dgcat \stackrel{(-)^{\pretr}}{\too} \dgcat \stackrel{F_1 \circ F}{\too} \dgcat_{ex,\aleph_1} \stackrel{U \circ U_1}{\too} \dgcat\,,
\end{equation}
where $(-)^{\pretr}$ denotes the Bondal-Kapranov's pre-triangulated envelope construction; see~\cite{BK} for details.
\begin{lemma}\label{lem:Infsums}
The above composed functor (\ref{eq:composition}) sends quasi-equiconic dg functors (\ref{def:quasiec}) between cofibrant dg categories to quasi-equivalences (\ref{def:quasieq}).
\end{lemma}
\begin{proof}
Note first that by \cite[Lemma\,2.13]{Thesis} $(-)^{\pretr}$ is a fibrant resolution functor with respect to the quasi-equiconic model structure; see Proposition~\ref{prop:fibquasiec}. Therefore, by~\cite[Prop.\,2.14]{Thesis}, $(-)^{\pretr}$ sends quasi-equiconic dg functors (between cofibrant dg categories) to quasi-equivalences (between cofibrant dg categories). Since $F\circ F_1$ is a left Quillen functor, quasi-equivalences between cofibrant dg categories are sent to weak equivalences in $\dgcat_{ex,\aleph_1}$. Finally, the composed functor $U \circ U_1$ maps weak equivalences to quasi-equivalences, and so the proof is finished.
\end{proof}
Thanks to Lemma~\ref{lem:Infsums} and \cite[Prop.\,8.4.8]{Hirschhorn}, the above composed functor (\ref{eq:composition}) admits a total left derived functor (with respect to the quasi-equiconic model structure), denoted by $$\cF(-): \Hec \too \Hec\,.$$
Notice that by construction, we have also a natural transformation $\mbox{Id} \Rightarrow \cF(-)$.
\begin{lemma}\label{lem:F-filt}
The functor $\cF(-)$ gives rise to a morphism of derivators $$\cF(-):\Trdgcat \too \Trdgcat\,,$$ which preserves $\aleph_1$-filtered homotopy colimits.
\end{lemma}
\begin{proof}
This follows from the fact that the functors $(-)^{\pretr}$ and $U\circ U_1$ preserve $\aleph_1$-filtered homotopy colimits. The case of $(-)^{\pretr}$ is clear. In what concerns $U\circ U_1$, please see~\cite[Prop.\,6.7]{Sums}.
\end{proof}
\begin{proposition}\label{prop:sums}
Let $\cA$ be a small dg category.
\begin{itemize}
\item[(i)] The small dg category $\cF(\cA)$ has countable sums.
\item[(ii)] The small dg category $\cF(\cA)$ is Morita fibrant (see Proposition~\ref{prop:fibMorita}).
\item[(iii)] The dg functor $\cA \to \cF(\cA)$ induces a fully-faithful triangulated functor (see Notation~\ref{not:tri&perf})
$$\tri(\cA) \to \tri(\cF(\cA))\simeq \dgHo(\cF(\cA))\,.$$ Moreover, the image of $\dgHo(\cA)$ in $\dgHo(\cF(\cA))$ forms a set of compact generators.
\end{itemize}
\end{proposition}
\begin{proof}
Condition (i) follows from \cite[Prop.\,3.8]{Sums}. Thanks to \cite[6.4]{Sums}, the dg category $\cF(\cA)$ is pre-triangulated (see Proposition~\ref{prop:fibquasiec}). Since in any triangulated category with countable sums (for instance $\h^0(\cF(\cA))$) every idempotent splits (see~\cite[Prop.\,1.6.8]{Neeman}), we conclude that $\cF(\cA)$ is Morita fibrant (see Proposition~\ref{prop:fibMorita}). Finally, condition (iii) follows from the construction of $\cF(-)$.  
\end{proof}
\begin{remark}\label{rk:compact}
By Proposition~\ref{prop:sums} (iii) and \cite[\S3.1, Lemma 2]{Marco} we have an equivalence
$$ \widetilde{\tri(\cA)} \stackrel{\sim}{\too} \dgHo(\cF(\cA))_c\,,$$
between the idempotent completion of $\tri(\cA)$ (see Notation~\ref{not:idemp}) and the triangulated subcategory of compact objects in $\dgHo(\cF(\cA))$.
\end{remark}
We can then associate to every small dg category $\cA$, an exact sequence
$$ \cA \too \cF(\cA) \too \cS(\cA):=\cF(\cA)/\cA\,,$$
where $\cF(\cA)/\cA$ denotes the Drinfeld's dg quotient~\cite{Drinfeld}. In this way, we obtain also a functor $$\cS(-):\Hec \too \Hec\,.$$
\begin{proposition}\label{prop:dg-setup}
The category $\Hec$ endowed with $\cF(-)$, $\cS(-)$ and with the functor (see Notation~\ref{not:tri&perf})
\begin{eqnarray*}
\tri: \Hec  \too  \Tri && \cA  \mapsto  \tri(\cA)\,,
\end{eqnarray*}
satisfies the hypothesis of Schlichting's set-up~\cite[\S\,2.2]{Marco}.
\end{proposition}
\begin{proof}
By construction, we have for every $\cA\in \Hec$ an exact sequence
$$ \cA \too \cF(\cA) \too \cS(\cA)\,.$$
Since $\cF(\cA)$ has countable sums, the Grothendieck group $\bbK_0(\cF(\cA))$ is trivial (see Notation~\ref{not:Kneg}). It remains to show that $\cF(-)$ and $\cS(-)$ preserve exact sequences. For this, consider the following diagram in $\Hec$
$$
\xymatrix{
\cA \ar[r]^I \ar[d] & \cB \ar[d] \ar[r]^P & \cC\ar[d] \\
\cF(\cA) \ar[d] \ar[r]^{\cF(I)} & \cF(\cB)\ar[d] \ar[r]^{\cF(P)} & \cF(\cC) \ar[d] \\
\cS(\cA) \ar[r]_{\cS(I)} & \cS(\cB) \ar[r]_{\cS(P)} & \cS(\cC)\,,
}
$$ 
where the upper horizontal line is a exact sequence. Notice that by Remark~\ref{rk:compact}, the triangulated functors $\tri(\cF(I))$ and $\tri(\cF(P))$ preserve compact objects. Moreover, since these functors preserve also  countable sums (see~\cite[6.4]{Sums}), \cite[\S3.1, Corollary 3]{Marco} 
implies that the middle line is also an exact sequence. Finally, by
\cite[\S3.1, Lemma 3]{Marco} and Remark~\ref{rk:compact}, we conclude that the lower line is also an exact sequence.
\end{proof}
\subsection{Relation with Frobenius pairs}\label{defFrobpairs}
Let us denote by $(\mbox{Frob}; \cF_{\cM}, \cS_{\cM})$ Schlichting's set-up on the category of Frobenius pairs; see \cite[\S5]{Marco} for details. Recall from
\cite[\S\,4.3, Def.\,6]{Marco} that we have a {\em derived category} functor
$$
\begin{array}{rcl}
D: \mbox{Frob} & \too & \Tri \\
\mathbf{A}=(\cA, \cA_0) & \mapsto & D\mathbf{A} = \underline{\cA}/\underline{\cA_0}\,,
\end{array}
$$
where $\underline{\cA}/\underline{\cA_0}$ denotes the Verdier's quotient of the stable categories.
Notice that we have a natural functor
$$
\begin{array}{rcl}
\mathbf{E}: \Hec & \too & \mbox{Frob} \\
\cA & \mapsto & (E(\cA), E(\cA)\mbox{-}\mbox{prinj})\,,
\end{array}
$$
where $E(\cA)=\perf^{\cW}(\cA) \subset \cC(\cA)$ (see Notation~\ref{not:Ktri}) is
the Frobenius category associated to the cofibrant $\cA$-modules which become isomorphic
in $\cD(\cA)$ to elements of $\tri(\cA)$, and
$E(\cA)\mbox{-}\mbox{prinj}$ is the full subcategory of projective-injective objects of $E(\cA)$.
Since
$$D\mathbf{E}(\cA)=\underline{E(\cA)}/\underline{E(\cA)\mbox{-}\mbox{prinj}} \simeq \tri(\cA)\,,$$
we have an essentially commutative diagram
$$
\xymatrix{
\Hec \ar[d]_{\tri} \ar[rr]^{\mathbf{E}} && \mbox{Frob} \ar[dll]^D \\
\Tri && \,. 
}
$$
\begin{proposition}\label{prop:comparison}
For any dg category $\cA$, we have canonical equivalences of triangulated
categories
$$D\cF_\cM\mathbf{E}(\cA)\simeq \tri(\cF\cA)\quad
\text{and}\quad D\cS_\cM\mathbf{E}(\cA)\simeq \tri(\cS\cA)\, .$$
\end{proposition}
\begin{proof}
We start by showing that $\mathbf{E} \circ \cF$ and $\cF_{\cM}\circ \mathbf{E}$ are weakly isomorphic. Let $\cA$ be a pre-triangulated dg category (see Proposition~\ref{prop:fibquasiec}) and $\Loc_S\cC(\cF\cA)$ the left Bousfiled localization of $\cC(\cF\cA)$ (see Notation~\ref{not:modules}), with respect to the set 
$$S:=\{\, \underset{i \in I}{\bigoplus}\, \underline{h}(x_i) \too \underline{h}(\underset{i \in I}{\bigoplus}\, x_i) \,| \, x_i \in \cF\cA\,,\,  |I| < \aleph_1 \,\}\,. $$
Let us now introduce an ``intermediate'' Frobenius pair $ \mathbf{E}'(\cF\cA):=(E'(\cF\cA), E'(\cF\cA)_0)$. The Frobenius category $E'(\cF\cA)\subset \Loc_S\cC(\cF\cA)$ consists of the cofibrant $\cF\cA$-modules which become isomorphic in $\Ho(\Loc_S \cC(\cF\cA))$ to representable ones, and $E'(\cF\cA)_0$ consists of the $\cF\cA$-modules $M$ whose morphism $M\to 0$ to the terminal object is a $S$-equivalence. Notice that since every representable $\cF\cA$-module is $S$-local
in $\cC(\cF\cA)$, we have a natural weak equivalence $ \mathbf{E}(\cF\cA) \stackrel{\sim}{\to} \mathbf{E}'(\cF\cA)$
of Frobenius pairs. Consider the following (solid) commutative diagram
$$
\xymatrix{
\cF_{\cM}(E(\cA)) \ar@{-->}[ddrr]_{\Phi} \ar@/^2pc/[rrrr]^{\mbox{colim}} && E(\cA) \ar[ll] \ar@{^{(}->}[rr] \ar[d] && \cC(\cA) \ar[d] \\
&& E(\cF\cA) \ar[d]^{\sim} \ar@{^{(}->}[rr] && \cC(\cF\cA) \ar[d] \\
&& E'(\cF\cA) \ar@{^{(}->}[rr] && \Loc_S\cC(\cF\cA)\,,
}
$$
where $\Phi$ is the composition
$ \cF_{\cM}(E(\cA)) \stackrel{\mbox{colim}}{\to} \cC(\cA) \to \Loc_S\cC(\cF\cA)$.
We obtain then the following commutative diagram of Frobenius pairs
$$
\xymatrix{
\mathbf{E}(\cA) \ar[d] \ar[r] & \mathbf{E} (\cF\cA) \ar[d]^{\sim} \\
\cF_{\cM}(\mathbf{E}(\cA)) \ar[r] & \mathbf{E}'(\cF\cA)\,.
}
$$
Thanks to \cite[\S\,5.2, Proposition 1]{Marco} and Proposition~\ref{prop:sums} (iii) the triangulated categories $\cD\cF_{\cM}(\mathbf{E}(\cA))$ and $\tri(\cF(\cA))\simeq \cD\mathbf{E}(\cF\cA)$ (and so $\cD\mathbf{E}'(\cF(\cA))$ are compactly generated (see \cite[\S\,8.1]{Neeman}) by $\tri(\cA)$. Therefore, the lower map in the above square is also a weak equivalence of Frobenius pairs. In conclusion, we obtain a (two steps) zig-zag of weak equivalences relating $\cF_{\cM}(\mathbf{E}(\cA))$ and $\mathbf{E}(\cF\cA)$. 

We now show that $\mathbf{E}\circ \cS$ and $\cS_{\cM} \circ \mathbf{E}$ are weakly isomorphic. For this, notice that we have the following induced zig-zag of weak equivalences of Frobenius pairs relating $\cS_{\cM}(\mathbf{E}(\cA))$ and $\mathbf{E}(\cS(\cA))$ \,:
$$
\xymatrix{
& (E(\cF\cA), S_0\mathbf{E}(\cF\cA)) \ar[r]^-{\sim} \ar[d]^-{\sim} & \mathbf{E}(\cS(\cA)) \\
\cS_{\cM}(\mathbf{E}(\cA)) \ar[r]^-{\sim} & (E'(\cF\cA), S_0\mathbf{E}'(\cF\cA)) & \,,
}
$$
where $S_0\mathbf{E}(\cF\cA)$ (resp. $S_0\mathbf{E}'(\cF\cA)$) is the full subcategory of $E(\cF\cA)$ (resp. of $E'(\cF\cA)$) consisting of those objects sent to zero in the quotient $\cD\mathbf{E}(\cF\cA)/\cD\mathbf{E}(\cA)$ (resp. in $\cD\mathbf{E}'(\cF\cA)/\cD\mathbf{E}(\cA)$). 
\end{proof}
\section{Non-connective $K$-theory}\label{chap:negK}
\subsection{Construction}
Let $\mathbf{V}=W^\op$, where $W$ is the poset $\{(i,j)\,|\,|i-j| \leq 1\} \subset \bbZ \times \bbZ$ considered as a small category (using the product partial order of $\bbZ \times \bbZ$). We can construct for every small dg category $\cA$, a diagram $\mbox{Dia}(\cA) \in \Trdgcat(\mathbf{V})$ as follows\,:
$$
\xymatrix@!0 @R=3.5pc @C=5pc{
&& & & \vdots & \\
&& & \cF\cS(\cA) \ar[r] &\cS^2(\cA) \ar[r] \ar[u] & \cdots \\
&& \cF(\cA) \ar[r] & \cS(\cA) \ar[u] \ar[r] & \cS(\cA)/\cS(\cA) \ar[u] &  \\
& \ast \ar[r] & \cA \ar[u] \ar[r] & \cA/\cA \ar[u] & & \\
\cdots \ar@{=}[r] & \ast \ar@{=}[r]  \ar@{=}[u] & \ast  \ar[u] &  & & \\
& \vdots \ar@{=}[u] & & & &
}
$$
Notice that the assignment $\cA \mapsto \mbox{Dia}(\cA)$ is functorial. We obtain then a morphism of derivators
$$ \mbox{Dia}(-): \Trdgcat \too \Trdgcat_{\mathbf{V}}\,.$$
\begin{lemma}\label{lem:preservesums}
Let $\cA$ be a pre-triangulated dg category (see Proposition~\ref{prop:fibquasiec}) with countable sums. Then for any dg category $\cB$, $\rep_{tr}(\cB, \cA)$ (see Proposition~\ref{prop:repec}) has countable sums.
\end{lemma}
\begin{proof}
Recall that $\cA$ has countable sums if and only if the diagonal functor
$$ \cA \too \underset{\mathbb{N}}{\prod}\,\cA$$
has a left adjoint. Since for any dg category $\cB$, $\rep_{tr}(\cB,-)$ is a $2$-functor which preserves (derived) products, the proof is finished.
\end{proof}
\begin{proposition}\label{prop:trivialK}
Let $\cA$ be a dg category and $n \geq 0$. Then $\cS^n(\cA)/\cS^n(\cA)$ and $\cF\cS^n(\cA)$
become trivial after application of $\Uaddw$ (see Theorem~\ref{thm:Uadd}).
\end{proposition}
\begin{proof}
Since $\cS^n(\cA)/\cS^n(\cA)$ is clearly isomorphic to the terminal object in $\Hec$ and $\Uaddw$ preserves the terminal object, $\Uaddw(\cS^n(\cA)/\cS^n(\cA))$ is trivial. In what concerns $\cF\cS^n(\cA)$, it is enough by Proposition~\ref{prop:GenAdd}, to show that for every $\alpha$-small dg cell $\cB$ (\ref{def:dgcell}), the spectrum (see Notation~\ref{not:Ktri})
$$ \bbR\Hom(\,\Uaddw(\cB),\,\Uaddw(\cF\cS^n(\cA)))\simeq \Kw\rep_{tr}(\cB, \cF\cS^n(\cA))$$
is (homotopically) trivial. By Proposition~\ref{prop:sums}\,(i), the dg category $\cF\cS^n(\cA)$ has countable sums, and so by Lemma~\ref{lem:preservesums} $\rep_{tr}(\cB, \cF\cS^n(\cA))$ has also countable sums. This implies, by additivity, that the identity of $\Kw\rep_{tr}(\cB, \cF\cS^n(\cA))$ is
homotopic to zero.
\end{proof}
\begin{notation}
Let $V(-)$ be the composed morphism
$$ \Trdgcat \stackrel{\mathrm{Dia}(-)^{\w}}{\too} \Trdgcat_{\mathbf{V}} \stackrel{\Uaddw}{\too} (\Maddw)_{\mathbf{V}}\,,$$
where $\mbox{Dia}(-)^{\w}$ is obtained from $\mbox{Dia}(-)$ by applying $(-)^{\w}$ objectwise (see Notation~\ref{not:hat}). Given any dg category $\cA$, Proposition~\ref{prop:trivialK} implies that $V(\cA)$ is a spectrum
in $\Maddw$; see \S\ref{sub:Stab}. We define $V_a(\cA)=\Omega^\infty V(\cA)$
as the infinite loop object associated
to $V(\cA)$. This construction defines a morphism of derivators
$$V_a:\Trdgcat\too\Maddw\, .$$
We also define $V_l$ as the composition $\gamma_! \circ V_a$.
$$V_l:\Trdgcat\too\Mlocvw\, .$$
\end{notation}
\begin{remark}\label{rk:description}
For any dg category $\cA$ we have\,:
$$ V_a(\cA)=\Omega^\infty V(\cA)
\simeq \underset{n \geq 0}{\hocolim}\,\Uaddw(\cS^n(\cA)^{\w})[-n]\, .$$
Furthermore, since we have a commutative diagram
$$
\xymatrix{
\Spec(\Maddw) \ar[d]_{\Omega^{\infty}} \ar[rr]^{\Spec(\gamma_!)} && \Spec(\Mlocvw) \ar[d]^{\Omega^{\infty}}\\
\Maddw \ar[rr]_{\gamma_!} && \Mlocvw\,,
}
$$
$V_l(\cA)$ can be described as
$$ V_l(\cA)= \underset{n \geq 0}{\hocolim}\,\,\Ulocvw(\cS^n(\cA)^{\w})[-n]\, .$$
Notice also that we have a natural $2$-morphism of derivators $ \Uaddw(-) \Rightarrow V_a(-)$.
\end{remark}

\begin{proposition}\label{compKnegdgFrob}
For any dg category $\cA$, there is a canonical isomorphism
$$\bbK(\cA)\simeq \underset{n \geq 0}{\hocolim}\,K(\cS^n(\cA))[-n]$$
in the stable homotopy category of spectra (see ~\ref{not:Kneg}).
\end{proposition}

\begin{proof}
This follows immediately from  the
very definiton of the non-connective $K$-theory spectrum
and from Proposition \ref{prop:comparison}.
\end{proof}

\begin{proposition}\label{prop:negK}
For any dg category $\cA$, we have a natural natural isomorphism in the stable homotopy category
of spectra (see Notation~\ref{not:Kneg})
$$ \bbR\Hom (\Ulocvw(\underline{k}), V_l(\cA)) \simeq \bbK(\cA)\,.$$
\end{proposition}
\begin{proof}
We already know from \ref{compKnegdgFrob}
that we have a weak equivalence of spectra 
$$\bbK(\cA)\simeq \underset{n \geq 0}{\hocolim}\,K(\cS^n(\cA))[-n]\, .$$
we have also a weak equivalence of spectra
$$\underset{n \geq 0}{\hocolim}\,K(\cS^n(\cA))[-n]
\simeq \underset{n \geq 0}{\hocolim}\,K(\cS^n(\cA)^{\w})[-n]\,,$$
(see the proof of \cite[\S12.1, Thm.\,8]{Marco}).
We deduce from these facts the following computations\,:
\begin{eqnarray}
\bbK(\cA)&\simeq &\underset{n \geq 0}{\hocolim}\,K(\cS^n(\cA)^{\w})[-n]\nonumber\\
&\simeq & \underset{n \geq 0}{\hocolim}\,\Kw\,\rep_{tr}(\underline{k},\cS^n(\cA)^{\w})[-n]
\nonumber \\
&\simeq &\underset{n \geq 0}{\hocolim}\,\bbR\Hom (\,\Ulocvw(\underline{k}),\, \Ulocvw(\cS^n(\cA)^{\w}))[-n] \label{number1} \\
& \simeq & \bbR\Hom (\,\Ulocvw(\underline{k}), V_l(\cA)) \label{number2}\, .
\end{eqnarray}
Equivalence~\eqref{number1} comes from Theorem \ref{thm:Kwloc},
and equivalence~\eqref{number2} from the compactness of $\Ulocvw(k)$ in $\Mlocvw$; see Theorem~\ref{thm:Kwloc} and Remark~\ref{rk:enrichcomp}.
\end{proof}

\subsection{Co-representability in $\Mlocw$}

\begin{proposition}\label{prop:Mor-V}
The morphism $V_l(-): \Trdgcat \to \Mlocvw$ sends derived Morita equivalences (\ref{def:Morita}) to isomorphisms.
\end{proposition}
\begin{proof}
As in the proof of Theorem~\ref{thm:weakloc}, it is enough to show
that $V_l(-)$ sends the maps of shape $\cA\to\cA^{\w}$ to isomorphisms. For this consider the following diagram
$$
\xymatrix{
\cA \ar[d]_{P} \ar[r] & \cF(\cA) \ar[d]^{\cF(P)} \ar[r] & \cS(\cA):=\cF(\cA)/\cA \ar[d]^{\cS(P)} \\
\cA^{\w} \ar[r] & \cF(\cA^{\w}) \ar[r] & \cS(\cA^{\w}):=\cF(\cA^{\w})/\cA^{\w}\,.
}
$$
Thanks to Proposition~\ref{prop:sums}, $\cF(P)$ is an isomorphism. Observe that the induced morphism $\cS(P)$, between the Drinfeld dg quotients, is also an isomorphism. Finally, using the description of $V_l(-)$ of Remark~\ref{rk:description}, the proof is finished.
\end{proof}
Notice that by Proposition~\ref{prop:Mor-V}, the morphism $V_l(-)$ descends to the derivator $\HO(\dgcat)$ associated to the Morita model structure.
\begin{proposition}\label{prop:V-loc}
Assume that $\alpha\geq\aleph_1$.
The morphism of derivators
$$ V_l(-): \HO(\dgcat) \too \Mlocvw$$
is an $\alpha$-localizing invariant (\ref{not:alphaLoc}).
\end{proposition}
\begin{proof}
We start by showing condition $\alpha$-flt). Since by Lemma~\ref{lem:F-filt}, the functor $\cF(-)$ preserves $\alpha$-filtered colimits, so it does (by construction) the functors $\cS^n(-),\, n \geq 0$. Now, Proposition~\ref{prop:hat-properties}\,(i), Theorem~\ref{thm:veryweakloc}, and the classical Fubini rule on homotopy colimits imply the claim.
Condition $\underline{\mbox{p}}$) is clear. In what concerns condition loc), let
$$ 0 \too \cA \stackrel{I}{\too} \cB \stackrel{P}{\too} \cC \too 0$$
be an exact sequence (\ref{def:ses}) in $\Hmo$. Consider the following diagram in $\Trdgcat(\mathbf{V})$
$$
\xymatrix{
\mbox{Dia}(\cA)^{\w} \ar@{=}[d] \ar[r] & \mbox{Dia}(\cB)^{\w} \ar@{=}[d] \ar[r] & \mbox{Dia}(\cA,\cB)^{\w}:= \mbox{Dia}(\cA)^{\w}/\mbox{Dia}(\cB)^{\w} \ar[d] \\
\mbox{Dia}(\cA)^{\w}\ar[r] & \mbox{Dia}(\cB)^{\w} \ar[r] & \mbox{Dia}(\cC)^{\w} \,,
}
$$
where $\mbox{Dia}(\cA,\cB)^{\w}$ is obtained by applying Drinfeld's dg quotient~\cite{Drinfeld} objectwise. Since the upper horizontal line is objectwise a strict exact sequence (\ref{def:Wald}), we obtain a distinguished triangle
$$V_l(\cA) \too V_l(\cB) \too \Omega^{\infty} \Ulocvw(\mbox{Dia}(\cA,\cB)^{\w}) \too V_l(\cA)[1]$$
in $\Mlocvw(e)$. We have the following explicit description of
$\Omega^{\infty} \Ulocvw(\mbox{Dia}(\cA,\cB)^{\w})$\,:
$$\Omega^{\infty} \Ulocvw(\mbox{Dia}(\cA,\cB)^{\w})=
\underset{n}{\hocolim}\, \Ulocvw(\cS^n(\cB)^{\w}/\cS^n(\cA)^{\w})[-n]\, .$$
We now show that the morphism
$$ D: \Ulocvw(\mbox{Dia}(\cA,\cB)^{\w}) \too  \Ulocvw(\mbox{Dia}(\cC)^{\w})$$
becomes an isomorphism after application of the infinite loop functor $\Omega^{\infty}$.
For this, consider the following solid diagram
$$ \xymatrix{
\cS^n(\cA)^{\w} \ar[r] \ar[d] & \cF\cS^n(\cA)^{\w} \ar[d] \ar[r] & \cS^{n+1}(\cA)^{\w} \ar[d] \\
\cS^n(\cB)^{\w} \ar[d]  \ar[r] & \cF\cS^n(\cB)^{\w} \ar[r] \ar[d] & \cS^{n+1}(\cB)^{\w} \ar[d] \\
\cS^n(\cB)^{\w}/\cS^n(\cA)^{\w} \ar[d] \ar[r] & \cF\cS^n(\cB)^{\w}/\cF\cS^n(\cA)^{\w} \ar[d]^{\phi} \ar[r] & \cS^{n+1}(\cB)^{\w}/\cS^{n+1}(\cA)^{\w} \ar[d] \\
\cS^n(\cC)^{\w} \ar[r] \ar@{-->}[ur]^{\psi_n} & \cF\cS^n(\cC)^{\w} \ar[r] & \cS^{n+1}(\cC)^{\w}\,,
}
$$  
where the composition in each line is trivial.
Since the dg functor $\cF\cS^n(\cA)^{\w} \to \cF\cS^n(\cB)^{\w}$ preserves countable sums,
\cite[\S3.1, Thm.\,2]{Marco} implies that the dg category $\cF\cS^n(\cB)^{\w}/\cF\cS^n(\cA)^{\w}$ is Morita fibrant (see Proposition~\ref{prop:fibMorita}) and so $\phi$ is an isomorphism. We obtain then a well defined morphism $\psi_n$, which induces an interpolation map
$$ \Psi_n: \Ulocvw(\cS^n(\cC)^{\w})[-n] \too \Ulocvw (\cS^{n+1}(\cB)^{\w}/ \cS^{n+1}(\cA)^{\w})[-(n+1)] \,.$$
Notice that we have commutative diagrams
$$
\xymatrix{
\Ulocvw(\mbox{Dia}(\cA,\cB)^{\w})_n[-n]  \ar[dd]_{D_n[-n]} \ar[rr] && \Ulocvw(\mbox{Dia}(\cA,\cB)^{\w})_{n+1}[-(n+1)] \ar[dd]^{D_{n+1}[-(n+1)]} \\
\\
\Ulocvw(\mbox{Dia}(\cC)^{\w})_n[-n] \ar[rr] \ar[uurr]^{\Psi_n} && \Ulocvw(\mbox{Dia}(\cC)^{\w})_{n+1}[-(n+1)]\, .
}
$$
The existence of such a set of interpolation maps $\{\Psi_n\}_{n \in \bbN}$
implies that the map
$$\underset{n}{\hocolim}\, \Ulocvw(\cS^n(\cB)^{\w}/\cS^n(\cA)^{\w})[-n]
\stackrel{\sim}{\too} \underset{n}{\hocolim}\, \Ulocvw(\cS^n(\cC)^{\w})[-n]
$$
is an isomorphism; see~\ref{lem:cofinality}. In other words,
the induced morphism
$$\Omega^{\infty}(D): \Omega^{\infty}\Ulocvw(\mbox{Dia}(\cA,\cB)^{\w}) \stackrel{\sim}{\too} \Omega^{\infty}\Ulocvw(\mbox{Dia}(\cC)^{\w})= V_l(\cC)$$
is an isomorphism.
\end{proof}

\begin{corollary}{(of Proposition~\ref{prop:V-loc})}\label{cor:V-loc}
If $\alpha\geq\aleph_1$, since $V_l(-)$ is an $\alpha$-localizing invariant,
Theorem~\ref{thm:weakloc} implies the existence of a morphism of derivators (which commutes with homotopy colimits)
$$ \mathsf{Loc}: \Mlocw \too \Mlocvw$$
such that $\mathsf{Loc}(\Ulocw(\cA))\simeq V_l(\cA)$, for every small dg category $\cA$.
\end{corollary}
\begin{proposition}\label{prop:can-isom}
If $\alpha\geq\aleph_1$, the two morphisms of derivators
$$ \mathsf{Loc}\, , \ l^{\ast}: \Mlocw \too \Mlocvw$$
are canonically isomorphic.
\end{proposition}
\begin{proof}
Consider the composed morphism $L:= \mathsf{Loc} \circ l_!$. Thanks to Theorem~\ref{thm:weakloc} and of Proposition \ref{prop:Mor-V},
the $2$-morphism of derivators $\Ulocvw \Rightarrow V_l(-)$ can be naturally extended
to a $2$-morphism $\eta: \mbox{Id} \Rightarrow L$.
We show first that the couple $(L,\eta)$ defines a left Bousfield
localization of the category $\Mlocvw(e)$, \ie we prove that
the natural transformations $L\eta$ and $\eta_L$ are equal isomorphisms.
By construction, the functor $L$ commutes with homotopy colimits, and so by Theorem \ref{thm:weakloc}
it is enough to prove it for the objects of shape $\Ulocvw(\cA)$,
with $\cA$ a (pre-triangulated) dg category. Thanks to Proposition~\ref{prop:Mor-V} we have $V_l(\cA)\simeq V_l(\cA^{\w})$, and Proposition~\ref{prop:V-loc}, applied recursively to the exact sequences
$$ \cS^n(\cA) \too \cF\cS^n(\cA) \too \cS^{n+1}(\cA),\, n \geq 0\,,$$
shows us that we have $V_l(\cS^n(\cA)) \simeq V_l(\cA)[n]$.
This implies the following isomorphisms\,:
$$
\begin{array}{rcl}
L^2(\Ulocvw(\cA)) & =& L(V_l(\cA)) \nonumber \\
& \simeq & \underset{n \geq 0}{\hocolim}\, V_l(\cS^n(\cA)^{\w})[-n] \\
& \simeq & \underset{n \geq 0}{\hocolim}\,  V_l(\cS^n(\cA))[-n] \\
& \simeq & \underset{n \geq 0}{\hocolim}\, (V_l(\cA)[n])[-n] \\
& \simeq &  V_l(\cA) \, .\\
\end{array}
$$
More precisely, a careful description of the above isomorphisms
shows that the morphisms
$$ \eta_{L(\Ulocvw(\cA))}, L(\eta_{\Ulocvw(\cA)}): L(\Ulocvw(\cA)){\too} L^2(\Ulocvw(\cA))$$
become equal isomorphisms after composing them with the canonical isomorphism
$$L^2(\Ulocvw(\cA))\stackrel{\sim}{\too}
\underset{m,n \geq 0}{\hocolim}\,(\Ulocvw(\cA)[m+n])[-m-n]\, .$$
Hence $ \eta_{L(\Ulocvw(\cA))}=L(\eta_{\Ulocvw(\cA)})$
is an isomorphism, which proves that $(L,\eta)$ defines a left Bousfield
localization of the category $\Mlocvw(e)$.

The general formalism
of left Bousfield localizations makes that we are now reduced to prove the
following property\,: a morphism of $\Mlocvw(e)$ becomes an isomorphism
after applying the functor $L$ if and only if it becomes an isomorphism
after applying the functor $l_!$.
For this purpose, it is even sufficient to prove
that the induced morphism $l_!\eta : l_! \Rightarrow l_!(L)$ is an isomorphism. Once again, it is enough to show it for the objects of the form $\Ulocvw(\cA)$. By virtue of Theorem~\ref{thm:weakloc}, the spectrum $ \{ \Ulocvw(\cS^n(\cA)^{\w}) \}_{n \geq 0}$
becomes an $\Omega$-spectrum in $\St(\Mlocw)$ after application of $l_!$. Therefore, we obtain an isomorphism $ \Ulocw(\cA) \simeq l_!V_l(\cA)$.
\end{proof}
\begin{remark}\label{rk:k-compact1}
Since the morphism of derivators $\mathsf{Loc}$ preserves homotopy colimits, Proposition~\ref{prop:can-isom} and Theorem~\ref{thm:Kwloc} imply that the object $\Ulocw(k)$ is compact (\ref{def:compact}) in the triangulated category $\Mlocw(e)$.
\end{remark}
\begin{theorem}\label{thm:negK}
Let $\cA$ be a small dg category. If $\alpha\geq\aleph_1$, we have a natural isomorphism in the stable homotopy category of spectra (see Notation~\ref{not:Kneg})
$$ \bbR\Hom(\,\Ulocw(\underline{k}),\, \Ulocw(\cA)\,) \simeq \bbK(\cA)\,.$$
\end{theorem}
\begin{proof}
We have the following isomorphisms\,:
\begin{eqnarray}
\bbR\Hom(\,\Ulocw(\underline{k}),\, \Ulocw(\cA)\,) & \simeq & \bbR\Hom(\,\Ulocvw(\underline{k}),\, l^{\ast}(\Ulocw(\cA))\,) \nonumber \\
& \simeq & \bbR\Hom(\,\Ulocvw(\underline{k}),\, \mathsf{Loc}(\Ulocw(\cA))\,) \label{number4}\\
& \simeq & \bbR\Hom(\,\Ulocvw(\underline{k}),\, V_l(\cA)\,) \label{number5} \\
& \simeq & \bbK(\cA) \label{number6}\\ \nonumber
\end{eqnarray}
Equivalence~\eqref{number4} comes from Proposition \ref{prop:can-isom},
equivalence \eqref{number5} follows from Corollary~\ref{cor:V-loc},
and equivalence~\eqref{number6} is Proposition~\ref{prop:negK}.
\end{proof}

\subsection{Filtered homotopy colimits}

Recall from \cite[\S 10]{Additive} the construction of the {\em universal localizing invariant}
$$ \Uloc : \HO(\dgcat) \too \Mloc\,.$$
\begin{theorem}{(\cite[Thm.\,10.5]{Additive})}\label{thm:loc}
The morphism $\Uloc$ \,:
\begin{itemize}
\item[flt)] commutes with filtered homotopy colimits;
\item[p)] preserves the terminal object and
\item[loc)] sends the exact sequences in $\Hmo$ (\ref{def:ses}) to distinguished triangles in $\Mlocw$.
\end{itemize}
Moreover $\Uloc$ is universal with respect to these properties, $\ie$ for every strong triangulated derivator $\bbD$, we have an equivalence of categories
$$ (\Uloc)^{\ast}: \HomC(\Mloc, \bbD) \stackrel{\sim}{\too} \HomL(\HO(\dgcat), \bbD)\,,$$
where the right-hand side consists of the full subcategory of $\uHom(\HO(\dgcat), \bbD)$ of morphisms of derivators which verify the above three conditions.
\end{theorem}
\begin{notation}\label{not:Loc}
The objects of the category $\HomL(\HO(\dgcat), \bbD)$ will be called {\em localizing invariants}.
\end{notation}
Note that $\Mloc=\Mlocwzero$, so that Theorem \ref{thm:loc} is a particular
case of Theorem \ref{thm:weakloc} with $\alpha=\aleph_0$.

We will consider now a fixed infinite regular cardinal $\alpha\geq\aleph_1$.
Thanks to Theorem~\ref{thm:weakloc}, we have a unique morphism of derivators
(which preserves all homotopy colimits)
$$ t_!: \Mlocw \too \Mloc\,,$$
such that $\Uloc = t_! \circ \Ulocw$.
\begin{proposition}\label{prop:locfiltered}
The morphism $t_!:\Mlocw \to \Mloc$ describes the derivator $\Mloc$ as the left Bousfiled localization of the derivator $\Mlocw$ by all maps of shape
\begin{equation}\label{filtered}
\underset{j \in J}{\mbox{hocolim}}\, \Ulocw(D_j) \too \Ulocw(\cA) \,,
\end{equation}
where $\cA$ is a dg category, $J$ is a filtered category, and $D: J \to \dgcat$ is a functor such that $\underset{j \in J}{\mbox{hocolim}}\, D_j\simeq \cA$ in $\Hmo$.
Moreover, the morphism $t_!$ has a fully-faithful right adjoint.
\end{proposition}
\begin{proof}
The first assertion follows directly from Theorem \ref{thm:loc} and from the universal property
of left Bousfield localizations of derivators; see Definition~\ref{def:Cis}.

In order to prove the second assertion, we start by replacing $\Mloc$ (resp. $\Mlocw$)
by $\Maddwzero$ (resp. by $\Maddw$).
We now describe the derivator $\Maddwzero$
as the left Bousfield localization of $\Maddw$ by an explicit small set of maps,
namely, the set $T$ of maps of shape
\begin{equation}\label{filteredproof}
\underset{j \in J}{\mbox{hocolim}}\, \Uaddw(D_j) \too \Uaddw(\cA) \,,
\end{equation}
where $\cA$ is an $\alpha$-small dg cell, $J$ is a $\alpha$-small filtered category, and $D: J \to \dgcat$ is
a functor with values in homotopically finitely presented dg categories, such
that $\underset{j \in J}{\mbox{hocolim}}\, D_j\simeq \cA$ in $\Hec$.

Notice that, by construction of $\Maddw$ (see Theorem~\ref{thm:Uadd}), for any triangulated strong derivator $\bbD$ there is a canonical equivalence of categories between $\HomC(\Maddw,\bbD)$
and the full subcategory of $\uHom(\dgcat_\alpha,\bbD)=\bbD(\dgcat^{\op}_\alpha)$ which consists of morphisms $\dgcat_\alpha\to\bbD$ which preserve the terminal object, send
derived Morita equivalences to isomorphisms, and send split exact sequences of $\alpha$-small dg cells
to split distinguished triangles in~$\bbD$.

The map $\dgcat_\alpha\to\Maddw$ corresponding to the identity of $\Maddw$
can be restricted to finite dg cells (\ie $\aleph_0$-small dg cells), hence defines a canonical
map $\dgcat_f\to\Maddw$ which preserves the terminal object, sends
derived Morita equivalences to isomorphisms, and sends split exact sequences of finite dg cells
to direct sums. In other words, it defines
a canonical homotopy colimit preserving map
$$\varphi:\Maddwzero\too\Maddw\, .$$
Let $\mathit{loc}_T:\Maddw\to\Maddwzero$ be the left Bousfield localization
of $\Maddw$ by $T$. Note that any $\alpha$-small dg cell is
a filtered union of its finite sub~dg~cells. In other words, for any
$\alpha$-small dg cell $\cA$, there exists an $\alpha$-small 
 filtered category $J$, a functor $D: J \to \dgcat$
with values in homotopically finitely presented dg categories, such
that $\underset{j \in J}{\mbox{hocolim}}\, D_j\simeq \cA$ in $\Hec$.
Therefore, by definiton of $\mathit{loc}_T$, we get the following essentially
commutative diagram of derivators
$$\xymatrix{
\dgcat_\alpha\ar[d]\ar[r]^{\Uaddw}\ar[d]&\Maddw\ar[drr]^{\mathit{loc}_T}&&\\
\dgcat\ar[r]_{\Uaddwzero}&\Maddwzero\ar[r]_{\varphi}&\Maddw\ar[r]_{\mathit{loc}_T}&
L_T\Maddw\, ,
}$$
from which we deduce that the morphism $\mathit{loc}_T\circ\Uaddw$ preserves
filtered homotopy colimits. This implies that the canonical map
(obtained from Theorem \ref{thm:Uadd} and from the universal property
of the left Bousfield localization by $T$)
$$L_T\Maddw\too\Maddwzero$$
is an equivalence of derivators. This in turns implies that
$\Mloc=\Mlocwzero$ is the left Bousfield localization of $\Mlocw$ by the set of maps of shape
\begin{equation}\label{filteredproof2}
\underset{j \in J}{\mbox{hocolim}}\, \Ulocw(D_j) \too \Ulocw(\cA) \,,
\end{equation}
where $\cA$ is an $\alpha$-small dg cell, $J$ is a $\alpha$-small filtered category, and $D: J \to \dgcat$ is a functor with values in homotopically
finitely presented dg categories, such that $\underset{j \in J}{\mbox{hocolim}}\, D_j\simeq \cA$ in $\Hmo$.
In particular, $\Mloc$ is obtained from $\Mlocw$ as a left Bousfield localization
by a small set of maps. Therefore it comes from a left Bousfield localization at the level of the underlying model categories (see \S\ref{def:Cis}). Hence the canonical
map $\Mlocw\to\Mloc$ has a fully-faithful right adjoint\footnote{The existence of this
fully-faithful right adjoint can be otained directly from the general theory of left Bousfield
localizations and from the fact that any homotopy colimit preserving morphism of derivators
which are obtained from combinatorial model categories has a right adjoint; see \cite{Renaudin}.
However, our proof is more explicit.}.
\end{proof}
We obtain then an adjunction
 $$
\xymatrix{
\Mlocw \ar@<-1ex>[d]_{t_!} \\
\Mloc \ar@<-1ex>[u]_{t^{\ast}}\,.
}
$$
\subsection{Co-representability in $\Mloc$}
\begin{theorem}\label{thm:main}
The object $\Uloc(\underline{k})$ is compact (\ref{def:compact}) in the triangulated category $\Mloc(e)$ and  for every dg category $\cA$,
we have a natural isomorphism in the stable homotopy category of spectra (see Notation~\ref{not:Kneg})
$$ \bbR\Hom(\,\Uloc(\underline{k}),\, \Uloc(\cA)\,) \simeq \bbK(\cA)\,.$$
\end{theorem}
\begin{proof}
The proof consists on verifying the conditions of Proposition~\ref{prop:genarg},
with $\bbD=\Mlocw$, $S$ the set of maps of shape (\ref{filteredproof2}),
$X= \Ulocw(\underline{k})$ and $M=\Ulocw(\cA)$. Since by
Remark~\ref{rk:k-compact1}, the object $\Ulocw(\underline{k})$ is compact
in the triangulated category $\Mlocw(e)$, it is enough to show that the functor
$$
\bbR\Hom(\Ulocw(\underline{k}), -): \Mlocw(e) \too \Ho(\Spt)
$$
sends the elements of $S$ to isomorphisms. This follows from
the fact that non-connective $K$-theory preserves filtered homotopy colimits
(see \cite[\S7, Lemma 6]{Marco}) and that, by virtue of Theorem~\ref{thm:negK}, we
have a natural isomorphism in the stable homotopy category of spectra
$$ \bbR\Hom(\,\Ulocw(\underline{k}),\, \Ulocw(\cA)\,) \simeq \bbK(\cA)\,,$$
for every dg category $\cA$.
\end{proof}
\begin{corollary}\label{cor:negK}
We have natural isomorphisms of abelian groups
$$ \Hom(\,\Uloc(\underline{k})[n],\, \Uloc(\cA)) \simeq \bbK_n(\cA)\,,\,\,n \in \bbZ\,.$$
\end{corollary}

\section{Higher Chern characters}\label{chap:Chern}
In this Section we will show how our co-representability Theorem~\ref{thm:main} furnishes us for free higher Chern characters and higher trace maps; see Theorem~\ref{thm:Chern}.
Recall that, throughout the article, we have been working over a commutative base ring $k$.

Consider a localizing invariant (\ref{not:Loc})
$$ E: \HO(\dgcat) \too \HO(\Spt) \,,$$
with values in the triangulated derivator of spectra.
Non-connective $K$-theory furnishes us also a localizing invariant\,:
$$ \bbK:\HO(\dgcat) \too \HO(\Spt)\, .$$
We will use Theorem \ref{thm:main} and its Corollary~\ref{cor:negK} to understand
the natural transformations from $\bbK$ to $E$.
Since $E$ and $\bbK$ are localizing invariants, they descend by Theorem~\ref{thm:loc} to homotopy colimit preserving morphisms of derivators
$$ E_{\mathsf{loc}},\, \bbK_{\mathsf{loc}}: \Mloc \too \HO(\Spt)\,,$$
such that $E_{\mathsf{loc}} \circ \Uloc=E$
and $\bbK_{\mathsf{loc}}\circ\Uloc=\bbK$.

\begin{theorem}\label{thm:abschern}
There is a canonical bijection between the following sets\,:
\begin{itemize}
\item[{(i)}] The set of $2$-morphisms of derivators $\bbK_{\mathsf{loc}} \Rightarrow E_{\mathsf{loc}}$.
\item[{(ii)}] The set of $2$-morphisms of derivators $\bbK\Rightarrow E$.
\item[{(iii)}] The set $E_0(\underline{k}):=\pi_0(E(\underline{k}))$ (see Definition~\ref{def:Icells}).
\end{itemize}
\end{theorem}

\begin{proof}
The bijection between (i) and (ii) comes directly from Theorem \ref{thm:loc}.
By virtue of Theorem \ref{thm:main}, we have
$$\bbK(\cA)\simeq\bbR\Hom(\,\Uloc(\underline{k}),\,\Uloc(\cA)\,)\, .$$
Therefore, as $E_{\mathsf{loc}}$ is a functor enriched in spectra (see \ref{sub:Stab}),
we have a natural map
$$\bbK(\cA)=\bbR\Hom(\,\Uloc(\underline{k}),\,\Uloc(\cA)\,)
\too \bbR\Hom(E_{\mathsf{loc}}(\Uloc(\underline{k})), E_{\mathsf{loc}}(\Uloc(\cA))\,)\, .$$
But, by definition of $E_{\mathsf{loc}}$ we also have:
$$\bbR\Hom(E_{\mathsf{loc}}(\Uloc(\underline{k})),E_{\mathsf{loc}}(\Uloc(\cA))\,)
\simeq \bbR\Hom(E(\underline{k}),E(\cA)\,)\simeq E(\cA)^{E(\underline{k})}\, .$$
In other words, we have a $2$-morphism of derivators
$$\bbK\Rightarrow F=\bbR\Hom(E(\underline{k}),E(-))\, .$$
Any class $x\in E_0(\underline{k})$, seen as a map $x:S^0\to E(\underline{k})$ in
the stable homotopy category of spectra, obviously defines a $2$-morphism of derivators
$F\Rightarrow E$ (using the formula $E=\bbR\Hom(S^0,E(-))$), hence gives
rise to a $2$-morphism of derivators $c_x:\bbK\too E$. On the other hand,
given a natural transformation $c:\bbK\Rightarrow E$, we have in particular a map
$$c(k):\bbK_0(k)=\pi_0\bbK(k)\too\pi_0E(\underline{k})=E_0(\underline{k})\,,$$
which furnishes us an element $c(k)(1)$ in $E_0(\underline{k})$.
One checks easily that the maps $c\mapsto c(k)(1)$ and $x\mapsto c_x$
are inverse to each other.
\end{proof}

\begin{remark}
As $\bbK_{\mathsf{loc}}$ and $E_{\mathsf{loc}}$ are enriched
over spectra (see \S\ref{sub:Stab}), there is a spectra $\bbR\mathsf{Nat}(\bbK_{\mathsf{loc}},E_{\mathsf{loc}})$
of natural transformations.
Theorem \ref{thm:abschern} and the enriched Yoneda lemma leads to
an enriched version of the preceding Theorem (as it was suggested in its proof)\,:
\end{remark}

\begin{corollary}
There is a natural isomorphism in the stable homotopy category of spectra
$$\bbR\mathsf{Nat}(\bbK_{\mathsf{loc}},E_{\mathsf{loc}})\simeq E(\underline{k})\, .$$
\end{corollary}

\begin{proof}
Theorem \ref{thm:abschern} implies that the natural map
$$\bbR\mathsf{Nat}(\bbK_{\mathsf{loc}},E_{\mathsf{loc}})\too
\overline{E}(\Uloc(\underline{k}))=E(\underline{k})$$
induced by the unit $1\in\bbK(k)$ leads to a bijection
$$\pi_0\bbR\mathsf{Nat}(\bbK_{\mathsf{loc}},E_{\mathsf{loc}})\simeq \pi_0E(\underline{k})$$
(as $\pi_0\bbR\mathsf{Nat}(\bbK_{\mathsf{loc}},E_{\mathsf{loc}})$ is the set
of $2$-morphisms $\bbK_{\mathsf{loc}}\Rightarrow E_{\mathsf{loc}}$). By applying
this to the morphism of derivators $E(-)[-n]:\HO(\dgcat) \to \HO(\Spt)$ (\ie $E$
composed with the loop functor iterated $n$ times), one deduces
that the map
$$\pi_n\bbR\mathsf{Nat}(\bbK_{\mathsf{loc}}, E_{\mathsf{loc}})\too \pi_n E(\underline{k})$$
is bijective for any integer $n$.
\end{proof}

As an illustration, let us consider
$$ \bbK_n(-): \Ho(\dgcat) \too \ModZ\,,\,\, n \in \bbZ\, ,$$
the $n$-th algebraic $K$-theory group functor (see~\cite[\S12]{Marco}),
$$ HC_j(-):\Ho(\dgcat) \too \ModZ\,,\,\, j \geq 0\, ,$$
the $j$-th cyclic homology group functor (see~\cite[Thm.\,10.7]{Additive}), and
$$ THH_j(-): \Ho(\dgcat) \too \ModZ\,,\,\, j \geq 0\, ,$$
the $j$-th topological Hochschild homology group functor (see~\cite[\S 11]{THH}).
\begin{theorem}\label{thm:Chern}
We have the following functorial morphisms of abelian groups for every
dg category $\cA$\,:
\begin{itemize}
\item[(i)] Higher Chern characters
$$ ch_{n,r}: \bbK_n(\cA) \too HC_{n+2r}(\cA)\,,\,\,n \in \bbZ\,,\,\, r \geq 0\,,$$
such that $ch_{0,r}: \bbK_0(k) \too HC_{2r}(k)$ sends $1\in \bbK_0(k)$ to
a generator of the $k$-module of rank one $HC_{2r}(k)$.
\item[(ii)] When $k=\bbZ$, higher trace maps
$$ tr_{n}: \bbK_n(\cA) \too THH_{n}(\cA)\,, \, n \in \bbZ\, ,$$
such that $tr_{0}: \bbK_0(k) \too THH_{0}(k)$ sends $1\in\bbK_0(\bbZ)$
to $1\in THH_0(\bbZ)$, and
$$ tr_{n,r}: \bbK_n(\cA) \too THH_{n+2r-1}(\cA)\,, \, n \in \bbZ\,,\,\, r \geq 1\,,$$
such that $tr_{0,r}: \bbK_0(k) \too THH_{2r-1}(k)$ sends $1\in\bbK_0(\bbZ)$
to a generator in the cyclic group $THH_{2r-1}(\bbZ)\simeq\bbZ/ r\bbZ$.
\item[(iii)] When $k=\bbZ/p\bbZ$, with $p$ a prime number, higher trace maps
$$ tr_{n,r}: \bbK_n(\cA) \too THH_{n+2r}(\cA)\,,\, n,r \in \bbZ\,,$$
such that $tr_{0,r}: \bbK_0(k) \too THH_{2r}(k)$ sends $1\in\bbK_0(\bbZ)$
to a generator in the cyclic group $THH_0(\bbZ/p\bbZ)\simeq\bbZ/ p\bbZ$.
\end{itemize}
\end{theorem}
\begin{proof}
(i) Recall from \cite[Thm.\,10.7]{Additive}, that we have a {\em mixed complex} localizing invariant
$$C: \HO(\dgcat) \too \HO(\Lambda\text{-}\mathsf{Mod})\,.$$
This defines a cyclic homology functor
$$HC: \HO(\dgcat) \too \cD(k)\, ,$$
where $HC(\cA)=k\otimes^\bbL_\Lambda C(\cA)$, and where $\cD(k)$
denotes the triangulated derivator associated to the
model category of unbounded complexes of $k$-modules (localized by
quasi-isomorphisms).
We thus have the following formula (see Lemma \ref{existSpHom})\,:
$$\begin{aligned}
HC_j(\cA)&=\Hom_{\cD(k)}(k,HC(\cA)[-j])\\
&\simeq\Hom_{\cD(k)}((S^0\otimes k,HC(\cA)[-j])\\
&\simeq\Hom_{\Ho(\Spt)}(S^0,\bbR\Hom_{\cD(k)}(k,HC(\cA)[-j]))\\
&\simeq\pi_0(\bbR\Hom_{\cD(k)}(k,HC(\cA)[-j]))\, .
\end{aligned}$$
Since $$ HC_{\ast}(k)\simeq k[u], \,\, |u|=2\,,$$
we conclude from Theorem \ref{thm:abschern} applied to
$E=\bbR\Hom_{\cD(k)}(k,HC(-))[-2r]$, that for each integer $r\geq 0$,
the element $u^r\in HC_{2r}(k)$ defines a natural map
$$\bbK(\cA)\too\bbR\Hom_{\cD(k)}(k,HC(\cA))[-2r]\, .$$

(ii) By Blumberg-Mandell's localization Theorem \cite[Thm.\,6.1]{BMandell} and the connection between dg and spectral categories developped in \cite{THH}, we have a localizing invariant $$THH : \HO(\dgcat) \too \HO(\Spt)\,.$$
Thanks to B{\"o}kstedt~\cite[0.2.3]{DGMcC}, we have the following calculation
\[
THH_j(\bbZ)=\left\{ \begin{array}{lcl}
\bbZ & \text{if} & j=0 \\
\bbZ/r\bbZ& \text{if} & j=2r-1, \,\, r \geq 1\\
0 & & \text{otherwise,}
\end{array} \right.
\]
Therefore the canonical generators of $\bbZ$ and $\bbZ/r\bbZ$ furnishes us the higher trace maps
by applying Theorem \ref{thm:abschern} to $E=THH(-)[-n]$ with $n=0$, or with $n=2r-1$ and $r\geq 1$.

(iii) The proof is the same as the preceding one\,: over the ring $\bbZ/p\bbZ$, we have the following calculation (see \cite[0.2.3]{DGMcC})
\[
THH_j(\bbZ/p\bbZ)=\left\{ \begin{array}{lcl}
\bbZ/p\bbZ & \text{if $j$ is even,} \\
0 & \text{if $j$ is odd.}
\end{array} \right.
\]
Hence the canonical generator $\bbZ/p\bbZ$ furnishes us the higher trace maps
by applying Theorem \ref{thm:abschern} to $E=THH(-)[-2r]$.

\end{proof}
\appendix
\section{Grothendieck derivators}\label{appendix}
\subsection{Notations}
The original reference for derivators is Grothendieck's
manus\-cript~\cite{Grothendieck} and Heller's monograph~\cite{Heller0}
on homotopy theories. See also \cite{maltsin,imdir,propuni,CN}.
A derivator $\bbD$ consists of a strict contravariant $2$-functor from the $2$-category of small
categories to the $2$-category of categories 
$$
\bbD: \Cat^{\op} \longrightarrow \CAT,
$$
subject to certain conditions, the main ones being that for any
functor between small categories $u:X\to Y$, the \emph{inverse image
functor}
$$u^*=\bbD(u):\bbD(Y)\too\bbD(X)$$
has a left adjoint, called the \emph{homological direct image functor},
$$u_!:\bbD(X)\too\bbD(Y)\, ,$$
as well as right adjoint, called the \emph{cohomological direct image functor}
$$u_{\ast}:\bbD(X)\too\bbD(Y)\,.$$
See~\cite{imdir} for details. 

\begin{example}
The essential example to keep in mind is the derivator $\bbD=\HO(\cM)$
associated to a (complete and cocomplete) Quillen model category~$\cM$
and defined for every small category~$X$ by
\begin{equation*}
\HO(\cM)\ (X)=\Ho\big(\Fun(X^{\op},\cM)\big)\,,
\end{equation*}
where $\Fun(X^\op,\cM)$ is the category of presheaves on $X$
with values in $\cM$, and $\Ho\big(\Fun(X^{\op},\cM)\big)$ is
its homotopy category; see \cite[Thm.\, 6.11]{imdir}.
\end{example}
We denote by $e$ the $1$-point category with one object and one
(identity) morphism. Heuristically, the category $\bbD(e)$ is the
basic ``derived" category under consideration in the
derivator~$\bbD$: we might think of $\bbD$ as a structure
on $\bbD(e)$ which allows to speak of homotopy
limits and colimits in $\bbD(e)$. For instance, if $\bbD=\HO(\cM)$ then
$\bbD(e)$ is the usual homotopy category $\Ho(\cM)$ of $\cM$.

\begin{definition}\label{def:srp}
We now recall some technical properties of derivators\,:
\begin{itemize}
\item[(i)] A derivator $\bbD$ is called {\em strong} if for every finite free category
$X$ and every small category $Y$, the natural functor $ \bbD(X
\times Y) \longrightarrow \Fun(X^{\op},\bbD(Y))$ is full and
essentially surjective. See~\cite{Heller} for details.
\item[(ii)]
A derivator $\bbD$ is called {\em regular} if in $\bbD$, sequential homotopy
colimits commute with finite products and homotopy pullbacks. See also~\cite{Heller} for details.
\item[(iii)] A derivator $\bbD$ is {\em pointed} if for any closed immersion $i:Z
\rightarrow X$ in $\Cat$ the cohomological direct image functor
$i_{\ast}:\bbD(Z) \longrightarrow \bbD(X)$ has a right adjoint, and
if moreover and dually, for any open immersion $j:U \rightarrow X$
the homological direct image functor $ j_!: \bbD(U) \longrightarrow
\bbD(X)$ has a left adjoint. See~\cite[1.13]{CN} for details.
\item[(iv)]
A derivator $\bbD$ is called {\em triangulated} or {\em stable} if it is
pointed and if every global commutative square in $\bbD$ is
cartesian exactly when it is cocartesian. See~\cite[1.15]{CN} for details.
\end{itemize}
\end{definition}
\begin{remark}
A strong derivator is the same thing as a small homotopy theory in
the sense of Heller~\cite{Heller}. By~\cite[Prop.\,2.15]{catder}, if
$\cM$ is a Quillen model category, its associated derivator
$\HO(\mathcal{M})$ is strong. Moreover, if sequential homotopy
colimits commute with finite products and homotopy pullbacks
in~$\cM$, the associated derivator $\HO(\cM)$ is regular. Notice also that if
$\cM$ is pointed, then so is~$\HO(\cM)$.  Finally, a pointed model category $\cM$ is stable
if and only if its associated derivator $\HO(\cM)$ is triangulated.
\end{remark}
\begin{notation}
Let $\bbD$ and $\bbD'$ be derivators. We denote by $\uHom(\bbD,\bbD')$ the
category of all morphisms of derivators. Its morphisms will be called $2$-morphisms
of derivators. Finally we denote by $\HomC(\bbD,\bbD')$ the
category of morphisms of derivators which commute with homotopy
colimits; see \cite{imdir,propuni}. For instance, 
any left (resp. right) Quillen functor induces a morphism of derivators
which preserves homotopy colimits (resp. limits); see \cite[Prop.\,6.12]{imdir}.
\end{notation}
\begin{notation}\label{derivatorpresheaves}
If $\bbD$ is a derivator and if $A$ is a small category, we denote by
$\bbD_A$ the derivator defined by $\bbD_A(X)=\bbD(A\times X)$.
One can think of $\bbD_A$ as the derivator of presheaves on $A$ with
values in $\bbD$.
\end{notation}
\begin{notation}\label{oppositederivator}
If $\bbD$ is a derivator, its \emph{opposite} $\bbD^\op$ is defined by
$\bbD^\op(A)=\bbD(A^\op)^\op$. In the case $\bbD=\HO(\cM)$, we thus
have $\bbD^\op=\HO(\cM^\op)$.
\end{notation}
\subsection{Left Bousfield localization}\label{sub:leftBousfield}
Let $\bbD$ be a derivator and $S$ a class of morphisms in the base category $\bbD(e)$.

\begin{definition}\label{def:Cis}
The derivator $\bbD$ admits a {\em left Bousfield localization} with
respect to~$S$ if there exists a morphism of derivators
$$ \gamma : \bbD \too \Loc_S\bbD\,,$$
which commutes with homotopy colimits, sends the elements of $S$ to
isomorphisms in $\Loc_S\bbD(e)$ and satisfies the following
universal property\,: for every derivator $\bbD'$ the morphism
$\gamma$ induces an equivalence of categories
$$\gamma^{\ast}: \HomC(\Loc_S\bbD,\bbD')
  \stackrel{\sim}{\too}
  \uHom_{!,S}(\bbD,\bbD')\,,$$
where $\uHom_{!,S}(\bbD,\bbD')$ denotes
the category of morphisms of derivators which commute with homotopy
colimits and send the elements of $S$ to isomorphisms in $\bbD'(e)$.
\end{definition}
\begin{theorem}{(\cite[Thm.\,4.4]{Additive})}
Let $\cM$ be a left proper, cellular model category and $\Loc_S\cM$
its left Bousfield localization \cite[Def.\,4.1.1]{Hirschhorn} with
respect to a set of morphisms~$S$ in $\Ho(\cM)$, $\ie$ to perform the localization we choose in $\cM$ a representative of each element of $S$. Then, the induced morphism of
derivators $\HO(\cM)\to \HO(\Loc_S\cM)$ is a left Bousfield
localization of derivators with respect to $S$.
In this situation, we have moreover a natural adjunction of derivators
$$
\xymatrix{
\HO(\cM) \ar@<1ex>[d] \\
\HO(\Loc_S\cM) \ar@<1ex>[u]\,.
}
$$
\end{theorem}
\begin{lemma}{(\cite[Lemma~4.3]{Additive})}
The Bousfield localization $\Loc_S\bbD$
of a \emph{triangulated} derivator $\bbD$ remains triangulated as
long as $S$ is stable under the desuspension functor $[-1]$.
\end{lemma}

\subsection{Stabilization and spectral enrichment}\label{sub:Stab}
Let $\bbD$ be a regular pointed strong derivator. In~\cite{Heller}, Heller constructed the universal morphism to a triangulated strong derivator
$$ \stab: \bbD \too \St(\bbD)\,,$$
in the sense of the following Theorem.
\begin{theorem}[Heller~\cite{Heller}]
Let $\bbT$ be a triangulated strong derivator. Then the morphism $\stab$ induces an equivalence of categories
$$ \stab^{\ast}:\HomC(\St(\bbD),\bbT) \too \HomC(\bbD,\bbT)\,.$$
\end{theorem}
The derivator $\St(\bbD)$ is described as follows.
Let $\mathbf{V}=W^\op$, where $W$ is the poset $\{(i,j)\,|\,|i-j| \leq 1\} \subset \bbZ \times \bbZ$ considered as a small category. A \emph{spectrum} in $\bbD$ is an object $X$ in
$\bbD(\mathbf{V})$ such that $X_{i,j}\simeq 0$ for $i\neq j$.
This defines a derivator $\Spec(\bbD)$ (as a full subderivator of $\bbD_{\mathbf{V}}$).
Spectra are diagrams of the following shape\,:
$$
\xymatrix@!0 @R=3.5pc @C=5pc{
&& & & \vdots & \\
&& & 0 \ar[r] &X_{i+2,i+2} \ar[r] \ar[u] & \cdots \\
&& 0 \ar[r] & X_{i+1,i+1} \ar[u] \ar[r] & 0 \ar[u] &  \\
& 0 \ar[r] & X_{i,i} \ar[u] \ar[r] & 0 \ar[u] & & \\
\cdots \ar[r] & X_{i-1,i-1} \ar[r]  \ar[u] & 0  \ar[u] &  & & \\
& \vdots \ar[u] & & & &
}
$$
In particular, we get maps $X_{i,i}\to\Omega(X_{i+1,i+1})$, $i\in\bbZ$.
The derivator $\St(\bbD)$ is obtained as the full subderivator
of $\Spec(\bbD)$ which consists of $\Omega$-spectra, \ie spectra $X$
such that the maps $X_{i,i}\to\Omega(X_{i+1,i+1})$ are all isomorphisms.

There is a morphism of derivators, called the \emph{infinite loop functor},
$$\Omega^\infty:\Spec(\bbD)\too\bbD$$
defined by
$$\Omega^\infty(X)=\underset{n}{\hocolim}\, \Omega^n(X_{n,n})$$
(where $\Omega^n$ denotes the loop space functor iterated $n$ times).
There is also a shift morphism
$$(-)<n>:\Spec(\bbD)\too\Spec(\bbD)$$
defined by $(X<n>)_{i,j}=X_{n+i,n+j}$ for $n\in\bbZ$.

The derivator $\St(\bbD)$ can be described as the left Bousfield localization
of $\Spec(\bbD)$ by the maps $X\to Y$ which induce isomorphisms
$\Omega^\infty(X<n>)\simeq\Omega^\infty(Y<n>)$ for any integer $n$; see \cite{Heller}.
Note that, for an $\Omega$-spectrum $X$, we have a canonical
isomorphism\,: $\Omega^\infty(X<n>)\simeq X_{n,n}$.
\begin{notation}
For a small category $A$, let $\mathsf{Hot}_A$ (resp. $\mathsf{Hot}_{\bullet,A}$)
be the derivator associated to the projective model category structure on the category
of simplicial presheaves (resp. of pointed simplicial presheaves) on $A$
(this is compatible with the notations introduced in \ref{derivatorpresheaves}).
\end{notation}
\begin{remark}\label{rempointedHot}
The homotopy colimit preserving morphism
$$\mathsf{Hot}_A\too\mathsf{Hot}_{\bullet,A}\ , \quad X\longmapsto X_+=X\amalg *$$
is the universal one with target a pointed regular strong derivators; see \cite[Prop. 4.17]{propuni}.
\end{remark}
\begin{notation}
Denote by $\Spt_A$ the stable Quillen model category of presheaves of spectra on $A$, endowed with
the projective model structure.
\end{notation} 

The infinite suspension functor
defines a homotopy colimit preserving morphism
$$\Sigma^\infty:\mathsf{Hot}_{\bullet,A}\too\HO(\Spt_A)\, ,$$
and as $\HO(\Spt_A)$ is a triangulated strong derivator, it induces
a unique homotopy colimit preserving morphism
$$\St(\mathsf{Hot}_{\bullet,A})\too\HO(\Spt_A)$$
whose composition with $\stab:\mathsf{Hot}_{\bullet,A}\to\St(\mathsf{Hot}_A)$
is the infinite suspension functor. A particular case of \cite[Thm.\,3.31]{Thesis}
gives\,:
\begin{theorem}The canonical morphism
$\St(\mathsf{Hot}_{\bullet,A})\to\HO(\Spt_A)$
is an equivalence of derivators. As a consequence, the map
$\Sigma^\infty:\mathsf{Hot}_{\bullet,A}\to\HO(\Spt_A)$
is the universal homotopy colimit preserving morphism
from $\mathsf{Hot}_{\bullet,A}$ to a triangulated derivator.
\end{theorem}
The composition of the Yoneda embedding $A\to\mathsf{Hot}_A$
with the infinite suspension functor gives a canonical morphism
$$h:A\too\HO(\Spt_A)\, .$$
A combination of the preceding theorem, of Remark \ref{rempointedHot},
and of \cite[Corollary 4.19]{propuni}, leads to the following statement.
\begin{theorem}\label{stableuniv}
For any triangulated derivator $\bbD$, the functor
$$h^*:\HomC(\HO(\Spt_A),\bbD)\stackrel{\sim}{\too}\uHom(A,\bbD)\simeq\bbD(A^\op)$$
is an equivalence of categories.
\end{theorem}
Hence, given any object $X$ in $\bbD(e)$, there is a unique
homotopy colimit preserving morphism of triangulated derivators
$$\HO(\Spt)\too\bbD\ , \quad E\longmapsto E\otimes X$$
which sends the stable $0$-sphere to $X$.
This defines a canonical action of $\HO(\Spt)$ on $\bbD$; see \cite[Thm. 5.22]{propuni}.
\begin{lemma}\label{existSpHom}
For any small category $A$, the functor
$$\Ho(\Spt_A)=\HO(\Spt)(A)\too\bbD(A)\ , \quad E\longmapsto E\otimes X$$
has a right adjoint
$$\bbR\Hom_{\bbD}(X,-):\bbD(A)\too\HO(\Spt)(A)\, .$$
\end{lemma}
\begin{proof}
This follows from the Brown representability Theorem
applied to the compactly generated triangulated category $\Ho(\Spt_A)$;
see \cite[Thm.\,8.4.4]{Neeman}.
\end{proof}
Using Theorem \ref{stableuniv}, one sees easily that the functors of
Lemma \ref{existSpHom} define a morphism
$$\bbR\Hom_{\bbD}(X,-):\bbD\too\HO(\Spt)$$
which is right adjoint to the morphism $(-)\otimes X$.
In particular, we have the formula
$$\Hom_{\Ho(\Spt)}(E,\bbR\Hom_{\bbD}(X,Y))\simeq\Hom_{\bbD(e)}(E\otimes X,Y)$$
for any spectrum $E$ and any objects $X$ and $Y$ in $\bbD(e)$.

This enrichment in spectra is compatible with homotopy colimit preserving morphisms
of triangulated derivators (see \cite[Thm. 5.22]{propuni})\,:
if $\Phi:\bbD\to\bbD'$ is a homotopy colimit preserving
morphism of triangulated derivators, then for any spectrum $E$ and any object $X$
of $\bbD$, we have a canonical coherent isomorphism
$$E\otimes\Phi(X)\simeq \Phi(E\otimes X)\, .$$
As a consequence, if moreover $\Phi$ has a right adjoint $\Psi$, then we have
canonical isomorphisms in the stable homotopy category of spectra\,:
$$\bbR\Hom_{\bbD'}(\Phi(X),Y)\simeq\bbR\Hom_{\bbD}(X,\Psi(Y))\, .$$
Indeed, to construct such an isomorphism, it is sufficient to
construct a natural isomorphism of abelian groups, for any spectrum $E$\,:
$$\Hom_{\Ho(\Spt)}(E,\bbR\Hom_{\bbD'}(\Phi(X),Y)))\simeq
\Hom_{\Ho(\Spt)}(E,\bbR\Hom_{\bbD}(X,\Psi(Y)))\, .$$
Such an isomorphism is obtained by the following computations\,:
$$\begin{aligned}
\Hom_{\Ho(\Spt)}(E,\bbR\Hom_{\bbD'}(\Phi(X),Y)))
&\simeq \Hom_{\bbD'}(E\otimes\Phi(X),Y)\\
&\simeq \Hom_{\bbD'}(\Phi(X\otimes E),Y)\\
&\simeq \Hom_{\bbD}(X\otimes E,\Psi(Y))\\
&\simeq \Hom_{\Ho(\Spt)}(E,\bbR\Hom_{\bbD}(X,\Psi(Y)))\, .
\end{aligned}$$
Note finally that we can apply Theorem \ref{stableuniv}
to $\bbD^\op$ (\ref{oppositederivator}): for any object $X$ of $\bbD$,
there is a unique homotopy colimit preserving morphism
$$\HO(\Spt)\too\bbD^\op\ , \quad E\longmapsto X^E$$
such that $X^{S^0}=X$. We then have the formula
$$\bbR\Hom_{\bbD}(X,Y^E)\simeq \bbR\Hom_{\Ho(\Spt)}(E,\bbR\Hom_\bbD(X,Y))
\simeq\bbR\Hom_\bbD(E\otimes X,Y)\, .$$

\subsection{The Milnor short exact sequence}

\begin{proposition}\label{Milnorsequence}
Let $F$ be a strong triangulated derivator, and $F_\bullet$
an object of $\bbD(\bbN^\op)$. We then have a distinguished triangle
$$\bigoplus_n F_n\xrightarrow{{\bf 1}-{\bf s}}\bigoplus_n F_n\too \hocolim F_\bullet
\too \bigoplus_n F_n[1]$$
where ${\bf s}$ is the morphism induced by the maps $F_n\to F_{n+1}$. 
As a consequence, we also have the Milnor short exact sequence
$$0\to\underset{n}{\mathrm{lim}}^1\Hom_{\bbD(e)}(F_n[1],A)\to\Hom_{\bbD(e)}(\hocolim F_\bullet,A)
\to\underset{n}{\mathrm{lim}}\, \Hom_{\bbD(e)}(F_n,A)\to 0\, .$$
\end{proposition}

\begin{proof}
Given an object $F_\bullet$ in $\bbD(\mathbb{N}^\op)$, we can consider
the \emph{weak homotopy colimit} of $F_\bullet$
(also called the sequential colimit of $F_\bullet$ in \cite[\S2.2]{HPS})
$$F'=\mathsf{cone}\big({\bf 1}-{\bf s}:\bigoplus_n F_n\too\bigoplus_n F_n\big)\, .$$
For any object $A$ of $\bbD(e)$, we then have a Milnor short exact sequence
$$0\too\underset{n}{\mathrm{lim}}^1\Hom_{\bbD(e)}(F_n[1],A)\too\Hom_{\bbD(e)}(F',A)
\too\underset{n}{\mathrm{lim}}\, \Hom_{\bbD(e)}(F_n,A)\too 0\, .$$
As the morphism $\bbR\Hom_{\bbD}(-,A)$ preserves homotopy limits,
we deduce from \cite[Thm.\, IX.3.1]{BKan} that we also have
a Milnor short exact sequence
$$0\to{\underset{n}{\mathrm{lim}}^1}\Hom_{\bbD(e)}(F_n[1],A)\to
\Hom_{\bbD(e)}(\underset{n}{\hocolim}\,  F_n,A)
\to\underset{n}{\mathrm{lim}}\Hom_{\bbD(e)}(F_n,A)\to 0\, .$$
We deduce from this a (non unique) isomorphism
$\underset{n}{\hocolim} \, F_n\simeq F'$.
\end{proof}

\begin{corollary}\label{lem:cofinality}
Let $\bbD$ a strong triangulated derivator, and
let $D_\bullet: X_{\bullet} \to Y_{\bullet}$ be a morphism in $\bbD(\mathbb{N}^\op)$.
If there exists a family of  interpolation maps $\{ \Psi_n \}_{n \in \mathbb{N}}$, making the diagram
$$
\xymatrix{
X_0 \ar[d]^{D_0} \ar[rr] && X_1 \ar[d]^{D_1} \ar[rr] && X_2 \ar[d]^{D_2} \ar[rr] && \cdots \\
Y_0 \ar[urr]_{\Psi_0} \ar[rr] && Y_1 \ar[urr]_{\Psi_1} \ar[rr] && Y_2 \ar[rr] \ar[urr]_{\Psi_2} && \cdots \\
}
$$
commutative in $\bbD(e)$, then the induced morphism
$$\hocolim D_\bullet:\underset{n}{\hocolim} \,X_{n}{\too} \underset{n}{\hocolim}\, Y_{n}$$
is an isomorphism.
\end{corollary}
\begin{proof}
This follows from the preceding proposition and from an easy cofinality argument.
See for instance \cite[Lemma 1.7.1]{Neeman}.
\end{proof}

\medbreak\noindent\textbf{Acknowledgments\,:} The authors would like to thank
the Institut Henri Poincar{\'e} in Paris, for its hospitality, where some of this work was carried out.

\end{document}